\documentclass[a4paper,11pt]{article}
\usepackage{graphicx}        
\usepackage{multicol}        
\usepackage[bottom]{footmisc}
\usepackage{epsfig}
\usepackage{authblk}
\usepackage{graphicx}
\usepackage{float,cancel}
\usepackage{diagbox}
\usepackage{nicefrac}
\usepackage[applemac]{inputenc}
\usepackage[T1]{fontenc}
\usepackage{blkarray}
\usepackage{mathtools,amsthm,amsfonts}

\usepackage{newtxtext}
\usepackage[upint,varg]{newtxmath}

\usepackage{color}

\usepackage{bm}
\usepackage{nicefrac}
\usepackage[hmargin={29mm,29mm},vmargin={30mm,35mm}]{geometry}
\usepackage{array}
\usepackage{pdfpages}
\usepackage{subfig}
\usepackage{multirow}
\usepackage{enumitem}
\usepackage{bbold}
\usepackage{xparse}								
\usepackage{stmaryrd}							
\usepackage{sidecap}								
\usepackage{upgreek}
\usepackage{accents}
\allowdisplaybreaks
\usepackage{tcolorbox}
\usepackage{algorithmic}
\usepackage[ocgcolorlinks,allcolors={blue}]{hyperref}
\usepackage{url}
\usepackage[bbgreekl]{mathbbol}

\DeclareSymbolFontAlphabet{\mathbb}{AMSb}
\DeclareSymbolFontAlphabet{\mathbbl}{bbold}

\def\N{\mathbb N}
\def\R{\mathbb R}

\def\P{\mathbb P}

\def\D{\mathcal{D}}

\def\<{\langle}
\def\>{\rangle}

\def\lsim{\lesssim}

\def\Chi{\raise .3ex \hbox{\large $\chi$}} 

\def\ov{\overline}

\newcommand{\HGbracket}[2]{\langle #1,#2\rangle_{\Gamma}}

\def\ol#1{\overline{#1}}

\def\[{\Bigl [}
\def\]{\Bigr ]}
\def\({\Bigl (}
\def\){\Bigr )}
\def\l{\iota}

\def\dsp{\displaystyle}
\def\x{{\bf x}}

\def\n{{\bf n}}

\def\U{{\bf U}}

\def\q{{\bf q}}

\def\g{{\mathbf{g}}}
\def\aa{{\mathfrak a}}

\def\faces{\mathcal{F}}

\def\x{{\bf x}}
\def\dsp{{\displaystyle x}}
\def\d{{\rm d}}
\def\div{{\rm div}}

\def\q{\mathbf{q}}
\def\K{\mathbb{K}}
\def\n{\mathbf{n}}
\def\dsp{\displaystyle}
\def\bu{\mathbf{u}}
\def\bv{\mathbf{v}}
\def\bw{\mathbf{w}}

\def\O{\Omega}
\def\G{\Gamma}

\newcommand{\CinfT}[2]{{C^\infty_{\partial#1,T}([0,T]\times(#1\backslash #2))}}
\newcommand{\bCinf}[2]{{\boldsymbol{C}^\infty_{\partial#1}(#1\backslash #2)}}

\newcommand{\jump}[1]{\llbracket #1 \rrbracket}

\newcommand{\Hhalfn}{H^{\nf12}_{0,j,n}(\Gamma)}
\newcommand{\Hminushalfn}{H^{-\nf12}_{0,j,n}(\Gamma)}

\def\bbsig{\bbsigma}
\def\bbeps{\bbespilon}
\newcommand{\bb}{\mathbb}

\newcommand{\dtk}{\delta t^{k+\frac{1}{2}}}
\newcommand{\dtkmun}{\delta t^{k-\frac{1}{2}}}

\newcommand{\nf}{\nicefrac}

\newcommand{\email}[1]{\href{mailto:#1}{#1}}

\begingroup
    \makeatletter
    \@for\theoremstyle:=definition,remark,plain\do{%
        \expandafter\g@addto@macro\csname th@\theoremstyle\endcsname{%
            \addtolength\thm@preskip\parskip
            }%
        }
\endgroup

\newtheorem{theorem}{Theorem}
\numberwithin{theorem}{section}
\newtheorem{proposition}[theorem]{Proposition}
\newtheorem{lemma}[theorem]{Lemma}

\theoremstyle{remark}
\newtheorem{remark}[theorem]{Remark}
\theoremstyle{definition}

\newtheorem{definition}[theorem]{Definition}

\newenvironment{acknowledgements}
      {\bigskip\bigskip~\newline\textbf{Acknowledgements}}

\def\g{{\rm nw}}
\def\l{{\rm w}}
\def\rt{{\rm rt}}

\def\trace{\gamma}
\def\grad{\nabla}
\def\del{\partial}
\def\div{{\rm div}}

\def\weakto{\rightharpoonup}

\definecolor{labelkey}{rgb}{0.6,0,1}

\usepackage{parskip}
\usepackage{algorithmic}

\usepackage{diagbox}
\usepackage{booktabs}

\makeatletter
\def\thm@space@setup{%
  \thm@preskip=\parskip \thm@postskip=0pt
}
\makeatother

\newcommand{\doi}[1]{\href{https://dx.doi.org/#1}{DOI:#1}}

\begin{document}
\title{
  Numerical analysis of a mixed-dimensional poromechanical model with frictionless contact at matrix--fracture interfaces
}
\author[1,2]{{Francesco Bonaldi}\footnote{Corresponding author, \email{francesco.bonaldi@univ-perp.fr}}}
\author[3,1]{{J\'er\^ome Droniou}\footnote{\email{jerome.droniou@umontpellier.fr}}}
\author[4]{{Roland Masson}\footnote{\email{roland.masson@univ-cotedazur.fr}}}
\affil[1]{IMAG, Univ Montpellier, CNRS, Montpellier, France}%
\affil[2]{LAMPS, Universit\'e de Perpignan Via Domitia, Perpignan, France}
\affil[3]{School of Mathematics, Monash University, Victoria 3800, Australia}%
\affil[4]{Universit\'e C\^ote d'Azur, Inria, CNRS, Laboratoire J.A. Dieudonn\'e, team Coffee, Nice, France}%
\date{}
\maketitle

\begin{abstract}
\noindent
We present a complete numerical analysis for a general discretization of a coupled flow--mechanics model in fractured porous media, considering single-phase flows and including frictionless contact at matrix--fracture interfaces, as well as nonlinear poromechanical coupling. Fractures are described as planar surfaces, yielding the so-called mixed- or hybrid-dimensional models. Small displacements and a linear elastic behavior are considered for the matrix.
The model accounts for discontinuous fluid pressures at matrix--fracture interfaces in order to cover a wide range of normal fracture conductivities.

The numerical analysis is carried out in the Gradient Discretization framework~\cite{gdm}, encompassing a large family of conforming and nonconforming discretizations. The convergence result also yields, as a by-product, the existence of a weak solution to the continuous model. A numerical experiment in 2D is presented to support the obtained result, employing a Hybrid Finite Volume scheme for the flow and second-order finite elements ($\P_2$) for the mechanical displacement coupled with face-wise constant ($\bb P_0$) Lagrange multipliers on fractures, representing normal stresses, to discretize the contact conditions. %
\end{abstract}
%
\section{Introduction}

Coupled flow and geomechanics in fractured porous media play an important role in many subsurface applications. This is typically the case of CO$_2$ storage, for which fault reactivation {(that is, the renewed displacement of a fault that has undergone a prolonged period of inactivity)} that can result from CO$_2$ injection must be avoided to preserve the storage integrity. On the other hand, in enhanced geothermal systems, fracture conductivity must be increased by hydraulic stimulation to produce heat while avoiding the risk of induced seismicity.
Such processes couple the flow in the porous medium and the fractures, the poromechanical deformation of the porous rock, and the mechanical behavior of the fractures. 
Their mathematical modeling is typically based on mixed-dimensional geometries representing the fractures as a network of codimension-one surfaces. The first ingredient of the model is a mixed-dimensional flow, typically coupling a Poiseuille flow along the network of fractures with the Darcy flow in the surrounding porous rock (called the matrix). This type of mixed-dimensional flow models has been introduced in the pioneering works \cite{MAE02,BMTA03,FNFM03,KDA04,MJE05}, and has been the object of an intensive research over the last twenty years, both in terms of discretization and numerical analysis \cite{RJBH06,MF07,ABH09,Jaffre11,TFGCH12,SBN12,BGGM2015,DHM16,BGGLM16,BHMS2016,AHMED201749,BHMS2018,BNY18,
FLEMISCH2018239,CDF2018,NBFK2019,BERRE2021103759,ANTONIETTI2021,AGHILI2021110452}. The second ingredient is based on poromechanics, coupling the rock deformation with the Darcy flow in the matrix domain, see \cite{coussy} as a reference textbook on this topic. In this paper we assume small strains and porosity variations, as well as a poroelastic mechanical behavior of the porous rock. The third ingredient of the fully coupled model is related to the mechanical behavior of the fractures, given the matrix mechanical deformation and the matrix and fracture fluid pressures. This behavior is typically based on mechanics governing the contact and slip conditions \cite{oden-contact,wriggers2006,Wohlmuth11}.
In this work, we restrict the analysis to a simplified contact model based on Signorini complementarity conditions, thereby assuming frictionless contact. Note also that fractures are assumed to pre-exist and that fracture propagation is not addressed here. 

The modeling and numerical simulation of such mixed-dimensional poromechanical problems have been the object of many recent works
\cite{tchelepi-castelletto-2020,GKT16,GH19,contact-norvegiens,thm-bergen,GDM-poromeca-cont,GDM-poromeca-disc,BDMP:21}. In terms of discretization, they are mostly based on conservative finite volume schemes for the flow and use either a conforming finite element method \cite{tchelepi-castelletto-2020,GKT16,GH19,BDMP:21} or a finite volume scheme for the mechanics \cite{contact-norvegiens,thm-bergen}. Note that the saddle point nature of the coupling between the matrix fluid pressure and the displacement field requires a compatibility condition between both discretizations, as in \cite{contact-norvegiens,thm-bergen,GDM-poromeca-cont,GDM-poromeca-disc,BDMP:21}, or alternatively additional stabilization terms to get the stability of the pressure solution in the limit of low rock permeability, incompressible fluid, and small times. 
The contact mechanics formulation is a key ingredient to efficiently handle the nonlinear variational inequalities of the contact fracture model. It is typically either based on a mixed formulation with Lagrange multipliers to impose the contact conditions as in \cite{tchelepi-castelletto-2020,contact-norvegiens,thm-bergen,BDMP:21}, or on a consistent Nitsche penalization method as in~\cite{GH19,beaude2023mixed}.  Again, if a mixed formulation is used, a compatibility condition must be satisfied between the Lagrange multiplier and displacement field spaces \cite{contact-norvegiens,thm-bergen,BDMP:21}, or a stabilization must be specifically designed \cite{tchelepi-castelletto-2020}.

Key difficulties arise in the analysis and numerical analysis of this type of mixed-dimensional poromechanical models. They are first related to the nonlinear dependence of the flow on the fracture aperture, and by the nonlinear contact mechanics itself. Second, they result from the geometrical complexity induced by fracture networks which must account for immersed, non-immersed, and intersecting fractures. Furthermore, one must also handle the degeneracy of the fracture conductivity as a result of the vanishing fracture aperture at immersed fracture tips. 
Few works are available in the literature and they are all based on simplifying assumptions, none addressing all the difficulties raised by such models. In~\cite{GWGM2015,hanowski2018}, the authors consider  a simplified model assuming open fractures with no contact and a frozen fracture conductivity leading to a linear poromechanical coupling. The well-posedness is obtained based on a detailed analysis of weighted function spaces accounting for the degeneracy of the fracture conductivity at the tips.
In~\cite{GDM-poromeca-cont,GDM-poromeca-disc}, a convergence analysis to a weak solution is carried out in the Gradient Discretization (GD) framework for a mixed-dimensional two-phase poromechanical model assuming open fractures with no contact but taking into account the full nonlinear coupling and the degeneracy at the fracture tips. As a result of the no-contact assumption, lower bounds on the discrete porosity and fracture aperture solutions must be assumed, which precludes to obtain the existence of a weak solution as a by-product of the convergence analysis.  
In \cite{Boon2023}, the authors carry out a well-posedness analysis based on mixed-dimensional geometries and monotone operators theory. They consider a Tresca frictional contact model, but neglect the dependence of the fracture conductivity on the aperture, and the conductivity is not allowed to vanish at the tips. {In the preprint~\cite{dehoop2023coupling}, the authors prove the well-posedness of a mixed-dimensional poromechanical model using a regularized contact model accounting for friction and assuming a frozen fracture conductivity. In the recent preprint~\cite{both2023global} the authors obtain a global existence result for a two-phase poromechanical model (without fractures) including a bound constraint on the porosity and accounting for  degeneracy of  mobilities.}

In this paper, we carry out a convergence analysis of the discrete model using compactness arguments.
The model we consider has a range of challenging features: frictionless contact at matrix--fracture interfaces, complex fracture networks, nonlinear dependence of the flow on the fracture aperture, as well as the degeneracy of the fracture conductivity at the fracture tips. The analysis we carry out moreover covers many different numerical approximations of the flow and mechanical components of the model. As a by-product of this convergence analysis, we also obtain the existence of a weak solution to the continuous model. 
To the best of our knowledge, this is the first complete numerical analysis of such a complex poromechanical model, with all the above-mentioned features. We also highlight how this analysis, carried out through a general approach, could potentially be adapted to even more complex models (e.g., multi-phase flows).

The model is discretized using the Gradient Discretization Method (GDM)~\cite{gdm}. This framework covers a wide range of conforming and nonconforming discretizations both for the flow and the mechanics. The discretization of the contact mechanics is, on the other hand, more constrained. It is based on a mixed formulation using piecewise constant Lagrange multipliers and assumes an inf-sup condition  between the space of Lagrange multipliers and the discrete displacement field~\cite{Wohlmuth11}. The choice of piecewise constant Lagrange multipliers has key advantages. It leads to a local expression of the contact conditions, allowing the use of efficient non-smooth Newton and active set algorithms, and requires no special treatment at fracture tips and intersections.

Our convergence proof elaborates on previous works. In \cite{BDMP:21} we proved stability estimates and existence of a solution for the GD of a mixed-dimensional two-phase poromechanical model with Coulomb frictional contact. In~\cite{GDM-poromeca-cont,GDM-poromeca-disc} we obtained compactness estimates for a mixed-dimensional two-phase poromechanical model with no contact.  Combining these results we can establish the relative compactness of the fracture aperture, which is a key element to pass to the limit in the discrete variational formulation of the flow model. The main new difficulty addressed in this work is related to passing to the limit in the mechanical variational inequality, which involves the Lagrange multiplier and the normal jump of the displacement field at matrix--fracture interfaces.

The remainder of the  paper is structured as follows. In Section~\ref{sec:notation}, we introduce the mixed-dimensional geometry and function spaces. In Section~\ref{sec:modeleCont}, we describe the continuous problem in its strong formulation, provide a definition of weak solutions, and list assumptions on the data. Section~\ref{sec:gradientscheme} presents the general GD framework adopted to define the discrete counterpart of the coupled problem. Section~\ref{sec:convergence} contains the convergence result and its proof, which is the main contribution of this work.  The possible extension of this convergence result to other models is then briefly discussed in Section~\ref{sec:other.models}. 
In Section~\ref{sec:num.example}, we present a 2D numerical test to numerically investigate the convergence of a scheme fitting our framework, based on a Hybrid Finite Volume method for the mixed-dimensional flow \cite{ecmor-joubine,BHMS2016}, and a $\mathbb{P}_2$--$\mathbb{P}_0$ discrete mixed formulation of the contact mechanics. Finally, in Section~\ref{sec:conclusions} we outline some concluding remarks.

\section{Mixed-dimensional geometry and function spaces}\label{sec:notation}
We distinguish scalar fields from vector fields by using, respectively, lightface letters and boldface letters. The overline notation $\ol{v}$ is used to distinguish an exact (scalar or vector) field from its discrete counterpart $v$. $\Omega\subset\R^d$, $d\in\{2,3\}$, is assumed to be a bounded polytopal domain, and is partitioned
into a fracture domain $\Gamma$ and a matrix domain $\O\backslash\overline\G$.
The network of fractures is defined by 
$$
\overline \Gamma = \bigcup_{i\in I} \overline \Gamma_i,
$$  
where each fracture $\Gamma_i\subset \Omega$, $i\in I$, {is a polygonal simply connected open domain of a plane of $\R^d$}. Without restriction of generality, we assume that the fractures may only intersect at their boundaries (Figure~\ref{fig:network}), that is, for any $i,j \in I$ with $i\neq j$, it holds $\Gamma_i\cap \Gamma_j = \emptyset$, but not necessarily $\overline{\Gamma}_i\cap \overline{\Gamma}_j = \emptyset$.

\begin{figure}[h!]
\begin{center}
\includegraphics[scale=.55]{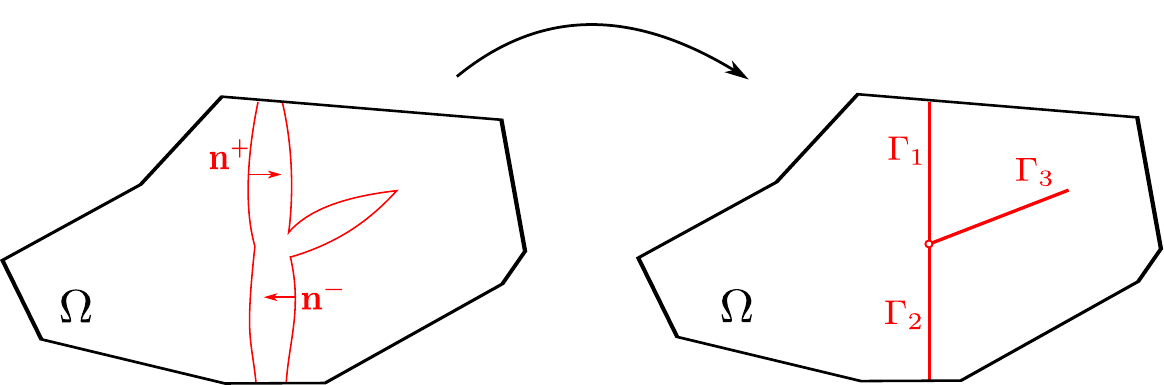}   
\caption{Illustration of the dimension reduction in the fracture aperture for a 2D domain $\Omega$ with three intersecting fractures $\Gamma_i$, $i\in\{1,2,3\}$. Equi-dimensional geometry on the left, mixed-dimensional geometry on the right.}
\label{fig:network}
\end{center}
\end{figure}

The two sides of a given fracture of $\Gamma$ are denoted by $\pm$ in the matrix domain, with unit normal vectors $\n^\pm$ oriented outward of the sides $\pm$. We denote by $\gamma_\aa$ the trace operator on side $\aa \in \{+,-\}$ of $\Gamma$ for functions in $H^1(\Omega{\backslash}\overline\Gamma)$ and by $\gamma_{\del\Omega}$ the trace operator for the same functions on $\del\O$. The jump operator on $\Gamma$ for functions $\ol{\bu}$ in  $H^1(\O\backslash\overline\Gamma)^d$ is defined by
$$
\jump{\ol{\bu}} = {\gamma_+\ol{\bu} - \gamma_-\ol{\bu}}, 
$$
and we denote by
$$
\jump{\ol{\bu}}_n = \jump{\ol{\bu}}\cdot\n^+
$$
its normal component. 
The tangential gradient and divergence along the fractures are respectively denoted by $\nabla_\tau$ and $\div_\tau$. The symmetric gradient operator $\bbeps$ is defined such that $\bbeps(\ol{\bv}) = {1\over 2} (\nabla \ol{\bv} +(\nabla \ol{\bv})^t)$ for a given vector field $\ol{\bv}\in H^1(\O\backslash\overline\Gamma)^d$.

\begin{figure}
  \begin{center}
  \includegraphics[scale=.65]{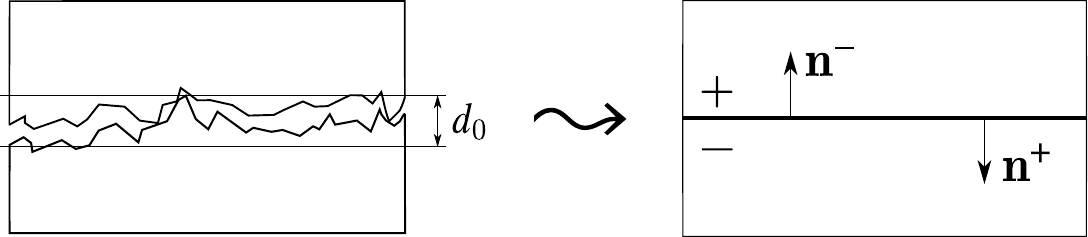}
  \caption{Conceptual fracture model with contact at asperities, $d_0$ being the fracture aperture at contact state.}
  \label{fig:d0}
\end{center}  
\end{figure}

Let us denote by $d_0: \Gamma \to (0,+\infty)$ the fracture aperture in the contact state (see Figure \ref{fig:d0}). The function $d_0$ is assumed to be continuous with zero limits at
$\partial \Gamma \backslash (\partial\Gamma\cap \partial\Omega)$ (i.e.~the tips of $\Gamma$) and strictly positive limits at $\partial\Gamma\cap \partial\Omega$.

Let us introduce some relevant function spaces. $H_{d_0}^1(\Gamma)$ is the space of functions $v_\Gamma\in L^2(\Gamma)$ such that $d_0^{\nf 3 2} \nabla_\tau v_\Gamma$ belongs to  $L^2(\Gamma)^{d-1}$, 
and whose traces are continuous at fracture intersections $\del\G_i\cap\del\G_j$ (for $(i,j)\in I\times I$ with $i\neq j$) and vanish on the boundary $\partial \Gamma\cap \partial\Omega$. The space for the displacement is
\begin{equation*}
\U_0 =\left\{ \ol{\bv}\in H^1(\O\backslash\overline\Gamma)^d : \trace_{\del\O} \ol{\bv} = 0\right\}.
\end{equation*}
The space for the pair of matrix/fracture pressures is
\begin{equation*}
V^0 = V^0_m \times V^0_f\quad \mbox{ with }\quad
V^0_m = \left\{ \ol{v}\in H^1(\Omega{\backslash}\overline\Gamma) \,:\, \trace_{\del\O} \ol{v} = 0\right\}\  \mbox{ and } \
V^0_f = H_{d_0}^1(\Gamma).
\end{equation*}
For $\ol{v} =(\ol{v}_m,\ol{v}_f) \in V^0$, the jump operator on side $\aa \in \{+,-\}$ of a given fracture is
$$
\jump{\ol{v}}_\aa = \gamma_\aa \ol{v}_m - \ol{v}_f.
$$
Finally, for any $x\in\mathbb R$, we set $x^+ = \max\{0,x\}$ and $x^- = (-x)^+$.

\section{Problem statement}\label{sec:modeleCont}

The primary unknowns of the coupled model in its strong form are the matrix and fracture pressures
$\ol{p}_\omega$, $\omega\in\{m,f\}$, and the displacement vector field $\ol{\bu}$. The problem  is formulated in terms of flow model and contact mechanics model, together with coupling conditions. The flow model is a mixed-dimensional model assuming an incompressible fluid and accounting for the volume conservation equations and for the Darcy law:
\begin{subequations}\label{eq:model}
\begin{equation}
\label{eq_edp_hydro} 
\left\{\hspace{-.3cm}
\begin{array}{lll}
&\partial_t \ol{\phi}_m   + \div \left(  \ol{\q}_m \right) = h_m & \mbox{ on } (0,T)\times \Omega{\backslash}\overline\Gamma,\\[1ex]
& \ol{\q}_m = \dsp -  {\K_m\over \eta}  \nabla \ol{p}_m & \mbox{ on } (0,T)\times \Omega{\backslash}\overline\Gamma,\\[1ex]
&\partial_t \ol{d}_f 
+ \div_\tau (  \ol{\q}_f ) - \ol{\q}_m\cdot\n^+ - \ol{\q}_m\cdot\n^-  = h_f & \mbox{ on } (0,T)\times \Gamma,\\[1ex]
& \ol{\q}_f = \dsp -   {1\over 12} {\ol{d}_f^{\,3} \over \eta}  \nabla_\tau \ol{p}_f  & \mbox{ on } (0,T)\times \Gamma. 
\end{array}
\right.
\end{equation}
In \eqref{eq_edp_hydro},  the constant fluid dynamic viscosity is denoted by $\eta>0$, the matrix porosity by $\ol{\phi}_m$ and the matrix permeability tensor by $\K_m$. The fracture aperture, denoted by $\ol{d}_f$, yields the fracture conductivity ${1\over 12} \ol{d}_f^{\,3}$ \emph{via} the Poiseuille law. 

The contact mechanics model accounts for the poromechanical equilibrium equation with a Biot linear elastic constitutive law and a frictionless contact model at matrix--fracture interfaces:
\begin{equation}
\label{eq_edp_meca} 
\left\{\hspace{-.3cm}
\begin{array}{lll}
& -\div \(\bbsig(\ol{\bu}) - b ~ \ol{p}_m{\mathbb I}\)= \mathbf{f} & \mbox{ on } (0,T)\times \Omega{\backslash}\overline\Gamma,\\[1ex]
  & \bbsig(\ol{\bu}) = \frac{E}{1+\nu}\(\bbeps(\ol{\bu}) + \frac{\nu}{1-2\nu} (\div\,\ol{\bu})\mathbb{I}\) & \mbox{ on } (0,T)\times \Omega{\backslash}\overline\Gamma,  \\[1ex]
  & \ol{\bf T}^{+} + \ol{\bf T}^{-} = {\bf 0}  & \mbox{ on } (0,T)\times \Gamma,\\[1ex]
  & \ol{T}_n \leq 0, \,\,   \jump{\ol{\bu}}_n \leq 0, \,\, \jump{\ol{\bu}}_n ~  \ol{T}_n = 0   & \mbox{ on } (0,T)\times \Gamma,\\[1ex]
 & \ol{\bf T}_\tau = {\bf 0}   & \mbox{ on } (0,T)\times \Gamma.\\[1ex]    
\end{array}
\right.
\end{equation}
In \eqref{eq_edp_meca}, $b$ is the Biot coefficient, $E$ and $\nu$ are the effective Young modulus and Poisson ratio, and the contact tractions are defined by 
\begin{equation*}
\left\{\hspace{-.3cm}
\begin{array}{lll}
  & \ol{\bf T}^{\aa} = \gamma^\aa_{\mathbf{n}}(\bbsig(\ol{\bu}) - b ~ \ol{p}_m \mathbb I) + \ol{p}_f \n^\aa  & \mbox{ on } (0,T)\times \Gamma,  \aa \in \{+,-\},\\[1ex]
  & \ol{T}_n = \ol{\bf T}^{+}\cdot \n^+ & \mbox{ on } (0,T)\times \Gamma, \\[1ex]
  & \ol{\bf T}_\tau = \ol{\bf T}^{+} - (\ol{\bf T}^{+}\cdot \n^+)\n^+ & \mbox{ on } (0,T)\times \Gamma, 
\end{array}
\right.
\end{equation*}
with $\gamma^\aa_{\mathbf{n}}$ the sided normal trace operator on the side $\aa\in \{+,-\}$ of $\Gamma$ for functions in $H_\div(\Omega\backslash\ov\Gamma)^{d\times d}$. 
The complete system of equations~\eqref{eq_edp_hydro}--\eqref{eq_edp_meca} is closed by means of coupling conditions. The first equation in \eqref{closure_laws} below accounts for the linear poroelastic state law for the variations of the matrix porosity $\ol{\phi}_m$. The second one stands for the matrix--fracture transmission conditions for the Darcy flow model. Following~\cite{FNFM03,MJE05}, they account for the normal flux continuity at each side $\aa$ of a fracture, combined with an approximation of the normal flux in the width of the fractures.  The normal fracture transmissibility $\Lambda_f$ is assumed here to be independent of $\ol{d}_f$ for simplicity, but the subsequent analysis can accommodate as well a fracture-width-dependent transmissibility, provided that the function $\ol{d}_f\mapsto\Lambda_f(\ol{d}_f)$ is continuous and bounded above and below by strictly positive constants.
The third equation in \eqref{closure_laws} is the definition of the fracture aperture $\ol{d}_f$.  
\begin{equation}
\label{closure_laws}
\left\{\hspace{-.3cm}
\begin{array}{lll}
  & \partial_t \ol{\phi}_m = \dsp b~\div \, \partial_t \ol{\bu} + \frac{1}{M} \partial_t \ol{p}_m & \mbox{ on } (0,T)\times \Omega{\backslash}\overline\Gamma,\\  [2ex]
  & \ol{\q}_m\cdot\n^\aa = \Lambda_f \jump{\ol{p}}_\aa  &   \mbox{ on } (0,T)\times \Gamma, \aa \in \{+,-\},\\[1ex]
& \ol{d}_f = d_0 -\jump{\ol{\bu}}_n   & \mbox{ on } (0,T)\times \Gamma.  
\end{array}
\right.
\end{equation}
\end{subequations}
The following initial conditions are imposed on the pressures and matrix porosity:
$$
(\ol{p}_\omega)|_{t=0} = \ol{p}_{0,\omega} \text{ for }\omega\in\{m,f\},  \quad (\ol{\phi}_m)|_{t=0}= \ol{\phi}_m^0, 
$$
and normal flux conservation for $\ol{\q}_f$ is prescribed at fracture intersections not located on the boundary $\partial\Omega$. 

As shown in Figure \ref{fig:d0}, due to surface roughness, the fracture aperture $\ol{d}_f \geq d_0$ does not vanish except at the tips. The open space is always occupied by the fluid, which exerts on each side $\aa$ of the fracture the pressure $\ol{p}_f$ appearing in the definition of the contact traction $\ol{\bf T}^{\aa}$.
{Note that the model is independent of the fracture orientation $\n^+$, since the normal traction $\ol{T}_n$ and the normal jump $\jump{\ol{\bu}}_n$, as well as the equation $\ol{\bf T}_\tau = {\bf 0}$, do not depend on this orientation.}

We make the following main assumptions on the data:
\begin{enumerate}[label=(H\arabic*),leftmargin=*]
\item $b\in [0,1]$ is the Biot coefficient, $M\in (0,+\infty]$ is the Biot modulus, and $E >0$, $-1<\nu<1/2$ are Young's modulus and Poisson's ratio, respectively. {Although the analysis easily extends to the case of non constant coefficients, they are assumed here to be constant for simplicity,} and $1/M$ is interpreted as $0$ when $M=+\infty$ (incompressible rock). \label{first.hyp}
\item The initial matrix porosity satisfies $\ol{\phi}_m^0 \in L^\infty(\O)$. 
  \item The fracture aperture at contact state satisfies $d_0 > 0$, and is continuous over the fracture network $\Gamma$ with zero limits at $\partial\Gamma\backslash (\partial\Gamma \cap \partial\Omega)$ and strictly positive limits at $\partial\Gamma\cap\partial\Omega$.
\item The initial pressures are such that $\ol{p}_{0,m}\in L^\infty(\Omega)$ and $\ol{p}_{0,f}\in L^\infty(\Gamma)$. 
\item The source terms satisfy $\mathbf f \in L^2(\Omega)^d$, $h_m\in L^2((0,T)\times \Omega),$ and $h_f \in L^2((0,T)\times \Gamma)$.
  \item The normal fracture transmissibility $\Lambda_f\in L^\infty(\Gamma)$ is uniformly bounded from below by a strictly positive constant. 
 \item The matrix permeability tensor $\K_{m}\in L^\infty(\Omega)^{d\times d}$ is symmetric and uniformly elliptic.
 \label{last.hyp}
\end{enumerate}

Following~\cite{Wohlmuth11}, the poromechanical model  with frictionless contact is formulated in mixed form using a scalar Lagrange multiplier $\ol{\lambda}: \Gamma \to \R$ at matrix--fracture interfaces. {We denote by
$$
\Hhalfn=\left\{\jump{\bv}_n\,:\,\bv\in \U_0\right\}\subset L^2(\Gamma)
$$
the space of normal jumps on $\Gamma$ of functions in $\U_0$, endowed with its natural norm
$$
\|q\|_{\Hhalfn}=\min\{\|\bv\|_{\U_0}\,:\,\bv\in \U_0,\,\jump{\bv}_n=q\}.
$$
We note that $\Hhalfn$ is not in general a subspace of $H^{\nf12}(\Gamma)$ since $\n$ is discontinuous at the intersections between transverse fractures. The dual of $\Hhalfn$ is denoted by $\Hminushalfn$, and the corresponding duality pairing is written $\HGbracket{\cdot}{\cdot}$.}
The dual positive cone is
$$
{C}_f = \left\{ \ol{\mu}  \in  {\Hminushalfn} \,:\,
 \HGbracket{\ol{\mu}}{\ol{v}} \leq 0 \mbox{ for all } \ol{v} \in {\Hhalfn} \mbox{ with } \ol{v} \leq 0\right\}. 
$$

The Lagrange multiplier formulation of \eqref{eq_edp_meca} then formally reads, dropping any consideration of regularity in time: find $\ol{\bu}:[0,T]\to \U_0$ and $\ol{\lambda} :[0,T]\to {C}_f$ such that for all $\ol{\bv}:[0,T]\to \U_0$ and $\ol{\mu}:[0,T]\to {C}_f$,
\begin{equation*}
\begin{aligned}
 & \dsp \int_\O \( \bbsig(\ol{\bu}): \bbeps(\ol{\bv}) - b ~\ol{p}_m \div(\ol{\bv})\) \d\x
+  \HGbracket{\ol{\lambda}}{\jump{\ol{\bv}}_n} 
+ \int_\G \ol{p}_f ~\jump{\ol{\bv}}_n  ~\d\sigma 
\dsp =  \int_\Omega \mathbf{f}\cdot \ol{\bv} ~\d\x,\\[.8em]
& \dsp \HGbracket{\ol{\mu}-\ol{\lambda}}{\jump{\ol{\bu}}_n} \leq 0.   
\end{aligned}
\end{equation*}
Note that, based on the variational formulation, the Lagrange multiplier satisfies $\ol{\lambda} = - \ol T_n$.

Let us denote by $\CinfT{\Omega}{\ol{\Gamma}}$ the space of smooth functions $\bar v:[0,T]\times (\O\backslash\overline\G)\to\R$ vanishing on $\partial\Omega$ and at $t=T$, and whose derivatives of any order admit finite limits on each side of $\Gamma$. We also denote, with an abuse of notation, 
$$
L^2(0,T;C_f)=\left\{\zeta\in L^2(0,T;{\Hminushalfn})\,:\,\zeta(t)\in C_f\mbox{ for a.e.~$t\in(0,T)$}\right\}.
$$
This set is endowed with the topology of $L^2(0,T;{\Hminushalfn})$.
The definition of weak solution to the model is the following.
\begin{definition}[Weak solution]\label{def:weak.solution} Under Assumptions~\ref{first.hyp}--\ref{last.hyp}, a weak solution to \eqref{eq:model} is a quadruple
$((\ol{p}_\omega)_{\omega\in\{m,f\!\}},\ol{\bu},\ol{\lambda})$ such that $(\ol{p}_m,\ol{p}_f)\in L^2(0,T;V^0)$, $\ol{\bu}\in L^\infty(0,T;\U_0)$, $\ol{d}_f^{\,\nf32}\nabla_\tau \ol{p}_f\in L^2((0,T)\times \Gamma)^{d-1}$ (with $\ol{d}_f=d_0-\jump{\ol{\bu}}_n$), $\ol{\lambda}\in L^2(0,T;C_f)$ and, for all $\varphi=(\varphi_m,\varphi_f)\in \CinfT{\Omega}{\ol{\Gamma}}\times C^\infty_c([0,T)\times \Gamma)$, $\ol{\bv}\in \U_0$, and $\ol{\mu}\in C_f$,
\begin{equation}
\begin{aligned}
&\int_0^T\int_\Omega \(-\ol{\phi}_m\partial_t\varphi_m + {\K_m \over \eta} \nabla\ol{p}_m\cdot\nabla\varphi_m\)~\d t\d\x
+\int_0^T\int_\Gamma \(-\ol{d}_f\partial_t\varphi_f+\frac{\ol{d}_f^{\,3}}{12\eta}\nabla_\tau\ol{p}_f\cdot\nabla_\tau\varphi_f\)~\d t\d\sigma\\
&\quad+\sum_{\aa\in\{+,-\}}\int_0^T\int_\Gamma\Lambda_f\jump{\ol{p}}^\aa\jump{\varphi}^\aa~\d t\d\sigma
{-}\int_\Omega \ol{\phi}_m^0\varphi_m(0)\d\x{-}\int_\Gamma \ol{d}_f^0\varphi_f(0)\d\x\\
&\qquad\qquad=\int_0^T\int_\Omega h_m\varphi_m~\d t\d\x
+\int_0^T\int_\Gamma h_f\varphi_f~\d t\d\sigma,
\label{weak:flow}
\end{aligned}\end{equation}
\begin{align}
&\int_\O \( \bbsig(\ol{\bu}): \bbeps(\ol{\bv}) - b ~\ol{p}_m \div(\ol{\bv})\) \d\x
+  \HGbracket{\ol{\lambda}}{\jump{\ol{\bv}}_n} 
+ \int_\G \ol{p}_f ~\jump{\ol{\bv}}_n  ~\d\sigma 
=  \int_\Omega \mathbf{f}\cdot \ol{\bv} ~\d\x\quad\mbox{ a.e. on $(0,T)$},
\label{weak:meca}\\[.8em]
& \HGbracket{\ol{\mu}-\ol{\lambda}}{\jump{\ol{\bu}}_n} \leq 0\quad\mbox{ a.e. on $(0,T)$}. 
\label{weak:multiplier}
\end{align}
The initial fracture aperture $\ol{d}_f^0$ in \eqref{weak:flow} is $\ol{d}_f^0=d_0-\jump{\ol{\bu}^0}_n$ where $(\ol{\bu}^0,\ol{\lambda}^0)\in\U_0\times C_f$ is the solution to \eqref{weak:meca}--\eqref{weak:multiplier} with $\ol{p}_{0,\omega}$ instead of $\ol{p}_\omega$ ($\omega\in\{m,f\}$).
These equations are complemented with the closure law
\begin{equation}\label{weak:closure.phi}
\ol{\phi}_m(t)-\ol{\phi}_m^0=b\,\div(\ol{\bu}(t)-\ol{\bu}^0)+\frac{1}{M}(\ol{p}_m(t)-\ol{p}_m^0)\quad\mbox{ for a.e. $t\in (0,T)$.}
\end{equation}
\end{definition}

\begin{remark}[Alternate formulation for the mechanical equations]\label{rem:meca.integral}
We note that \eqref{weak:meca}--\eqref{weak:multiplier} are equivalent to imposing these equations only at $t=0$, which fixes the initial displacement, and: for all $\ol{\bv}\in {L^2}(0,T;\U_0)$ and all $\ol{\mu}\in L^2(0,T;C_f)$,
\begin{align}
&\int_0^T\int_\O \( \bbsig(\ol{\bu}): \bbeps(\ol{\bv}) - b ~\ol{p}_m \div(\ol{\bv})\) ~\d t\d\x
+ \int_0^T \HGbracket{\ol{\lambda}}{\jump{\ol{\bv}}_n} \d t
+ \int_0^T\int_\G \ol{p}_f ~\jump{\ol{\bv}}_n  ~\d t\d\sigma \nonumber\\
&\qquad\qquad
=  \int_0^T\int_\Omega \mathbf{f}\cdot \ol{\bv} ~\d t\d\x,
\label{weak:meca.alternate}\\
& \int_0^T\HGbracket{\ol{\mu}-\ol{\lambda}}{\jump{\ol{\bu}}_n}~\d t \leq 0. 
\label{weak:multiplier.alternate}
\end{align}
\end{remark}

\section{The Gradient Discretization Method}\label{sec:gradientscheme}
\subsection{Gradient discretizations}
The Gradient Discretization for the Darcy discontinuous pressure model, introduced in~\cite{DHM16}, is defined by a finite-dimensional vector space of discrete unknowns $X^0_{\D_p} = X^0_{\D_p^m} \times X^0_{\D_p^f}$, 
and:
\begin{itemize}
\item  two discrete gradient linear operators on the matrix and fracture domains
$$
 \nabla_{\D_p}^m : X^0_{\D_p^m} \rightarrow L^\infty(\Omega)^d, \quad \quad 
   \nabla_{\D_p}^f : X^0_{\D_p^f} \rightarrow L^\infty(\Gamma)^{d-1},
$$ 
\item two function reconstruction linear operators on the matrix and fracture domains 
$$
\Pi_{\D_p}^m : X^0_{\D_p^m} \rightarrow L^\infty(\Omega),\quad\quad \Pi_{\D_p}^f : X^0_{\D_p^f} \rightarrow L^\infty(\Gamma),
$$
\item for each $\aa \in \{+,-\}$, a jump reconstruction linear operator $\jump{\cdot}^\aa_{\D_p}$: $X^0_{\D_p} \rightarrow L^\infty(\Gamma)$.
\end{itemize}  
The vector space $X^0_{\D_p}$ is endowed with the following quantity, assumed to define a norm:
\begin{equation}\label{def:XDp.norm}
\|v\|_{\D_p}:=   \|\nabla_{\D_p}^m v\|_{L^2(\Omega)^d} 
+ \|d_{0}^{\nf 3 2}\nabla_{\D_p}^f v\|_{ L^2(\Gamma)^{d-1} } + \sum_{\aa \in \{+,-\}} \|\jump{v}^\aa_{\D_p}\|_{L^2(\Gamma)}.
\end{equation}

As usual in the GDM framework~\cite[Part III]{gdm}, various choices of these spaces and operators lead to various numerical methods for the flow component of the model. Schemes covered by this framework include Hybrid Finite Volume (HFV) methods~\cite{ecmor-joubine,BHMS2016}, cell-centered finite volume schemes with Two-Point Flux Approximation (TPFA) on strongly admissible meshes \cite{KDA04,ABH09,gem.aghili}, or some symmetric Multi-Point Flux Approximations (MPFA) \cite{TFGCH12,SBN12,AELHP153D} on tetrahedral or hexahedral meshes; Mixed Hybrid Mimetic and Mixed or Mixed Hybrid Finite Element discretizations \cite{MJE05,BHMS2016,AFSVV16,Girault2019}; and vertex-based discretizations such as the Vertex Approximate Gradient scheme \cite{BHMS2016,DHM16,BHMS2018}. For the discretization of the mechanical component of the system, we however restrict ourselves to conforming methods for the displacement (usual methods for elasticity models), and piecewise constant spaces for the Lagrange multipliers.  We therefore take a finite-dimensional space
$$
X_{\D_\bu} = X^0_{\D_\bu^m}\times  X_{\D_\bu^f} \ \mbox{ with } \ X^0_{\D_\bu^m}\subset\U_0
 \mbox{ and } X_{\D_\bu^f} = \left\{\mu=(\mu_\sigma)_{\sigma\in\faces_{\D_\bu}}\,:\,\mu_\sigma\in \R\right\},
$$
where $\faces_{\D_\bu}$ is a partition of $\Gamma$ assumed to be conforming with the partition $\{\Gamma_i, i\in I\}$ of $\Gamma$ in planar fractures, {in the sense that each fracture $\Gamma_i$ is the union of a subset of faces of $\faces_{\D_\bu}$}. We also identify $\mu\in X_{\D_\bu^f}$ with the piecewise constant function $\mu:\Gamma\to\R$ defined by $\mu_{|\sigma}=\mu_\sigma$ for all $\sigma\in \faces_{\D_\bu}$. The discrete dual positive cone is defined as
$$
{C}_{\D_\bu^f} = \left\{\mu \in X_{\D_\bu^f}  \,:\, \mu_{\sigma}\geq 0,  \quad\forall \sigma \in\faces_{\D_\bu}\right\}.  
$$

A spatial GD can be extended into a space-time GD by complementing it with:
\begin{itemize}
\item 
a discretization $ 0 = t_0 < t_1 < \dots < t_N = T $ of the time interval $[0,T]$,
\item
interpolators of initial conditions: $I^m_{\D_p} \colon L^\infty(\Omega)\rightarrow X^0_{\D_p^m}$ and $I^f_{\D_p} \colon L^\infty(\Gamma)\rightarrow X^0_{\D_p^f}$ for the pressures, and $J^m_{\D_p} \colon L^2(\Omega)\rightarrow X^0_{\D_p^m}$ for the porosity.
\end{itemize}
For $k\in\{0,\ldots,N\}$, we denote by $\dtk = t_{k+1}-t_k$ the time steps, and by $\Delta t = \max_{k\in\{0,\ldots,N\}} \dtk$ the maximum time step. 

Spatial operators are extended into space-time operators as follows. Let $\Psi_\D$ be a spatial GDM operator defined in $X_\D^0$ with $\D=\D_p^m$ or $\D_p^f$,  and let $w=(w_k)_{k=0}^{N}\in (X^0_{\D})^{N+1}$. Then, its space-time extension is defined by 
$$
\Psi_{\D}w(0,\cdot) = \Psi_{\D}w_0, \mbox{ and } \Psi_{\D}w(t,\cdot) = \Psi_{\D}w_{k+1} \mbox{ for all $t\in (t_k,t_{k+1}]$ and $k\in\{0,\dots,N-1\}$}.
$$
For convenience, the same notation is kept for the spatial and space-time operators. Similarly, we identify $(X_{\D_\bu^m}^0)^{N+1}$ and $(X_{\D_\bu^f})^{N+1}$, respectively, with the spaces of piecewise constant functions $[0,T]\to X_{\D_\bu^m}^0$ and $[0,T]\to X_{\D_\bu^f}$; so, for example, if $\bu=(\bu^k)_{k\in\{0,\ldots,N\}}\in (X_{\D_\bu^m}^0)^{N+1}$, we set $\bu(0)=\bu^0$ and $\bu(t)=\bu^{k+1}$ for all $t\in (t_k,t_{k+1}]$ and $k\in\{0,\ldots,N-1\}$.
Moreover, if $E$ is a vector space and $f:[0,T]\to E$ is piecewise constant on the time discretization with $f_k=f_{|(t_{k-1},t_k]}$ and $f_0=f(0)$, the discrete time derivative of $f$ is
\begin{equation}\label{def_deltat}
\delta_t f (t) = \frac{f_{k+1} - f_k}{\dtk}\mbox{ for all  $t\in (t_k,t_{k+1}]$, $k\in\{0,\ldots,N-1\}$}. 
\end{equation}
This discretization leads to the implicit Euler time stepping in the following gradient scheme formulation. 
  
\subsection{Gradient scheme}

For $\sigma \in \faces_{\D_\bu}$, the displacement average on each side of $\sigma$, and the displacement normal jump average, are defined by:
$$
{\bf u}_\sigma^\aa = {1\over |\sigma|}\int_\sigma \gamma_\aa{\bf u}(\x)  \d\sigma \quad (\aa\in\{+,-\}),\quad
\jump{\bu}_{n,\sigma}  =  {\bf u}_\sigma^+\cdot {\bf n}^+ + {\bf u}_\sigma^-\cdot {\bf n}^-.
$$
The global displacement normal jump reconstruction $\jump{\bu}_{n,\faces}: \G\to\R$ is defined such that, for any $\sigma\in\faces_{\D_\bu}$, ${(\jump{\bu}_{n,\faces})}_{|_\sigma} = \jump{\bu}_{n,\sigma}$.

The gradient scheme for \eqref{eq:model} consists in writing a discrete weak formulation obtained, after a formal integration by parts in space, by replacing the continuous operators by their discrete counterparts: find $p = (p_m, p_f) \in (X^0_{\D_p})^{N+1}$, $\bu \in (X^0_{\D_\bu^m})^{N+1}$, and $\lambda\in ({C}_{\D_\bu^f})^{N+1}$, such that
\begin{subequations}\label{eq:GS}
\begin{equation}
\begin{aligned}
  {}&\int_0^T \int_\Omega \( (\delta_t \phi_\D)  \Pi_{\D_p}^m \varphi_m 
  +  {\K_m \over \eta} \nabla_{\D_p}^m p_m \cdot   \nabla_{\D_p}^m \varphi_m \) ~\d t\d\x 
   + \int_0^T \int_\Gamma (\delta_t d_{f,\D_\bu})\Pi_{\D_p}^f \varphi_f ~\d t\d\sigma\\
   &+ \int_0^T \int_\Gamma  \dsp {d_{f,\D_\bu}^{\,3} \over 12 \eta} \,\, \nabla_{\D_p}^f p_f \cdot   \nabla_{\D_p}^f \varphi_f  ~\d t\d\sigma 
   + \sum_{\aa \in \{+,-\}}
  \int_0^T \int_\G   \Lambda_f \jump{p}^\aa_{\D_p}   \jump{\varphi}^\aa_{\D_p}  ~\d t\d\sigma  \\
  & =  \int_0^T \int_\Omega h_m \Pi_{\D_p}^m \varphi_m ~\d t\d\x  + \int_0^T \int_\Gamma h_f \Pi_{\D_p}^f \varphi_f ~\d t\d\sigma\qquad\forall \varphi = (\varphi_m, \varphi_f) \in (X_{\D_p}^0)^{N+1},
  \label{GD_hydro}
\end{aligned}
\end{equation}
\medskip
\begin{equation}
\begin{aligned}
     \int_\Omega \Big( \bbsig(\bu^k) : \bbeps(\bv)
    -  {}& b ~\Pi_{\D_p}^m p^k_m~  \div\,\bv\Big)  \d\x 
   +  \int_\Gamma   \lambda^k ~ \jump{\bv}_n~\d\sigma\\
     &
     +  \int_\Gamma  \Pi_{\D_p}^f p^k_f~  \jump{\bv}_{n,\faces} ~\d\sigma 
  = \int_\Omega  \mathbf{f} \cdot  \bv ~\d\x \qquad\forall \bv \in X^0_{\D_\bu^m}\,,\;\forall k\in\{0,\ldots,N\},
\end{aligned}
\label{GD_meca}
\end{equation}
\begin{equation}
     \int_\Gamma (\mu - \lambda^k) \jump{\bu^{k}}_{n}~ \d\sigma \leq 0\qquad \forall\mu \in {C}_{\D_\bu^f}\,,\forall k\in\{0,\ldots,N\}, 
  \label{GD_meca_var}
\end{equation}
with the closure equations
\begin{equation}
\begin{aligned}
  &  \phi_{\D} - \Pi_{\D_p}^m  \phi_m^0 = b ~\div(\bu-\bu^0) + {1\over M} \Pi_{\D_p}^m (p_m-p_m^{0}),\\
  & d_{f,\D_\bu} = d_0 - \jump{\bu}_{n,\faces}, 
\end{aligned}
  \label{GD_closures}
\end{equation}
\end{subequations}
{from which $\delta_t \phi_\D$ and $\delta_t d_{f,\D_\bu}$ are deduced using the definition \eqref{def_deltat} of the discrete time derivative}. 
The initial conditions on the pressures and porosity are taken into account by setting $p^0_{\omega} = I^\omega_{\D_p} \ol{p}_{0,\omega}$ ($\omega\in \{m,f\}$), $\phi_m^0 = J_{\D_p}^m \ol{\phi}^0$. We note that \eqref{GD_meca}--\eqref{GD_meca_var} could be equivalently written using integrals over time (considering $\bv=\bv^k$, multiplying by $\dtk$ and summing over $k$), provided the equations on the initial displacement and multiplier, corresponding to $k=0$ in \eqref{GD_meca}--\eqref{GD_meca_var}, remain imposed separately.

We note the following local reformulation of the variational condition \eqref{GD_meca_var}, which corresponds to \cite[Lemma 4.1]{BDMP:21} with friction coefficient $F=0$.

\begin{lemma}[Local contact conditions]\label{lem:local.coulomb}
Let $\lambda \in (C_{\D_\bu^f})^{N+1}$ and $\bu\in (X_{\D^m_\bu}^0)^{N+1}$. Then $(\bu,\lambda)$ satisfy the variational inequality \eqref{GD_meca_var} if and only if the following local contact conditions hold on $[0,T]$ and for any $\sigma\in \faces_{\D_\bu}$: $\lambda_{\sigma} \geq 0$, $\jump{\bu}_{n,\sigma} \leq 0$ and $\jump{\bu}_{n,\sigma}  \lambda_{\sigma} = 0$.
\end{lemma}

\section{Convergence analysis}\label{sec:convergence}

\subsection{Assumptions and statement of the result}\label{sec:assum.theorem}

We assume in the following analysis that the gradient discretizations for the flow are \emph{coercive} in the sense of \cite{GDM-poromeca-disc} but without reference to the discrete trace operator (which is not used in the single-phase model). This means that the norm \eqref{def:XDp.norm} satisfies the following uniform Poincar\'e inequality: there exists $c^\star>0$ independent of $\D_p$ such that, for all $v=(v_m,v_f)\in X_{\D_p}^0$,
\begin{equation}\label{eq:GD.coercif}
\|\Pi^m_{\D_p} v_m\|_{L^2(\Omega)} + \|\Pi^f_{\D_p} v_f \|_{L^2(\Gamma)} 
\le c^\star\|v\|_{\D_p}. 
\end{equation}

The fracture network is assumed to be such that the Korn inequality holds on $\U_0$ (which is the case if the boundary of each connected component of $\Omega{\backslash}\Gamma$ has an intersection with $\partial\Omega$ that has a nonzero measure, see e.g.~\cite[Section~1.1]{ciarlet}); this ensures that the following expression defines a norm on $\U_0$, which is equivalent to the $H^1$-norm:
$$
\|\bv\|_{\U_0} = \|\bbeps(\bv)\|_{L^2(\Omega,\R^{d\times d})}.
$$

We also assume that $X_{\D_\bu^m}^0$ satisfies the following discrete inf-sup condition, in which the infimum and supremum are taken over nonzero elements of $X_{\D_\bu^f}$ and $X_{\D_\bu^m}^0$, respectively, and $c_\star$ does not depend on the mesh:
\begin{equation}\label{infsup}
\inf_{\boldsymbol{\mu}} \sup_{\bv} {\int_\Gamma \mu  ~\jump{\bv}_n \over \|{\bv}\|_{\U_0} \|\mu\|_{{\Hminushalfn}}} \geq c_{\star}  > 0. 
\end{equation}
We note that this inf-sup condition holds if $\D_\bu^m$ corresponds to the conforming $\P_1$ bubble or $\P_2$ finite elements on a regular triangulation of $\Omega\backslash\Gamma$, and $\faces_{\D_\bu}$ is made of the traces on $\Gamma$ of this triangulation, see \cite{BR2003} and the discussion in \cite[Section 4.3]{BDMP:21}. 

\begin{theorem}[Convergence of the gradient scheme]\label{th:convergence}
Let $(\D_p^l)_{l\in\N}$, $(X^l_{\D_\bu})_{l\in\N}$, and $\{(t^l_n)_{n=0}^{N^l}\}_{l\in\N}$ be sequences of gradient discretizations, contact mechanics conforming spaces, and time steps. We assume that the coercivity and inf--sup inequality \eqref{eq:GD.coercif}--\eqref{infsup} hold uniformly with respect to $l$, and that these sequences are consistent and limit-conforming as per \cite[Section 3.1]{GDM-poromeca-disc} (disregarding the properties associated with the reconstructed trace operators), and satisfy the following compactness property in the matrix:
\begin{equation}\label{eq:compact.matrix}
\begin{aligned}
&\text{For all $(v^l)_{l\in\N}$ with $v^l\in X^0_{\D_p^l}$ and $(\|v^l\|_{\D_p^l})_{l\in\N}$ bounded,}\\
&\text{$(\Pi^m_{\D_p^l}v^l)_{l\in\N}$ is relatively compact in $L^2(\Omega)$.}
\end{aligned}
\end{equation}
Then for each $l$ there exists a solution $(p^l,\bu^l,\lambda^l)$ to the scheme \eqref{eq:GS} with $(\D_p,X_{\D_\bu},(t_n)_{n=0,\ldots,N})=(\D_p^l,X^l_{\D_\bu},(t^l_n)_{n=0}^{N^l})$ and, along a subsequence as $l\to+\infty$,
\begin{subequations}\label{eq:convergences}
\begin{alignat}{3}
  &\Pi^m_{\D_p^l} p^l_m \weakto \ol{p}_m &\quad& \mbox{weakly in } L^2(0,T;L^2(\Omega)), \label{eq:conv.pm}\\
  &\nabla^m_{\D_p^l} p^l_m \weakto \nabla\ol{p}_m &\quad& \mbox{weakly in } L^2(0,T;L^2(\Omega))^d, \label{eq:conv.grad.pm}\\
  &\Pi^f_{\D_p^l} p^l_f \weakto \ol{p}_f &\quad& \mbox{weakly in } L^2(0,T;L^2(\Gamma)), \label{eq:conv.pf}\\
  &d_{f,\D_\bu^l}^{\,\nf32}\grad_{\D_p}^f p_f^l\weakto \ol{d}_f^{\,\nf32}\nabla_\tau\ol{p}_f &\quad& \mbox{weakly in } L^2(0,T;L^2(\Gamma)^{d-1}), \label{eq:conv.gradpf}\\
  &\jump{p^l}^\aa_{\D_p^l}\weakto \jump{\ol{p}}^\aa &\quad& \mbox{weakly in }L^2((0,T)\times\Gamma), \label{eq:conv.jump.p}\\
  &\bu^l \weakto \ol{\bu} &\quad& \mbox{weakly-$\star$ in } L^\infty(0,T;\U_0),\label{eq:conv.u}\\
  &\phi_{\D^l} \weakto \ol{\phi}_m &\quad& \mbox{weakly-$\star$ in } L^\infty(0,T;L^2(\Omega)),\label{eq:conv.phi}\\
  &d_{f,\D_\bu^l} \rightarrow \ol{d}_f &\quad& \mbox{in } L^\infty(0,T;L^p(\Gamma)) \mbox{ for all } 2 \leq p < 4,\label{eq:conv.df}\\
  &\lambda^l \weakto \ol{\lambda} &\quad& \mbox{weakly in } L^2(0,T;{\Hminushalfn}),\label{eq:conv.lambda}
\end{alignat}
\end{subequations}
where $(\ol{p}=(\ol{p}_m,\ol{p}_f),\ol{\bu},\ol{\lambda})$ is a weak solution to \eqref{eq:model} in the sense of Definition \ref{def:weak.solution} (with $\ol{\phi}_m$, $\ol{d}_f$ given in this definition).
\end{theorem}

\begin{remark}[GDM properties and existence of a solution to the weak formulation]
As discussed in \cite{gdm,GDM-poromeca-disc}, many numerical methods -- including those used for the tests in Section \ref{sec:num.example} (Hybrid Finite Volumes for the flow and a conforming finite-element method for the mechanics) -- fit the GDM framework and satisfy the required properties.
As a consequence of the proof of Theorem \ref{th:convergence}, which does not require to assume the existence of a solution to the continuous problem \eqref{eq:model}, we infer the existence of a weak solution to this model.
\end{remark}

In the following, we denote $a\lsim b$ for $a\le Cb$ with $C$ depending only on the data, $c^\star$ and $c_\star$, but not depending on $l$ or the considered functions.
For legibility, we also often drop the explicit mention of the index $l$ in the sequences.

\subsection{Energy estimates and existence of a solution}

The following energy estimates and existence of a solution to the scheme can be established as in~\cite[Theorem 4.2 and Theorem 4.4]{BDMP:21}, which correspond to the more challenging situation of a two-phase flow and a Coulomb friction contact model; fixing, in this reference, the friction coefficient $F$ to $0$, and the wetting saturation $S^{\rm w}$ to $1$ we recover the single-phase frictionless model \eqref{eq:model}.

\begin{theorem}[Energy estimates for \eqref{eq:GS}]
If $(p,\bu,\lambda)$ solves the gradient scheme \eqref{eq:GS}, then there exists $C\ge 0$ depending only on the data in Assumptions~\ref{first.hyp}--\ref{last.hyp} (except the Biot coefficient $b$ and the Biot modulus $M$), and on $c^\star,c_\star$, such that
\begin{equation}\label{apriori.est}
\begin{aligned}
\|\grad_{\D_p}^m p_m\|_{L^2((0,T)\times\Omega)} \le C, 
&\quad& \| d_{f,\D_\bu}^{\nf 3 2} \grad_{\D_p}^f p_f\|_{L^2((0,T)\times\Gamma)} \le C, \\
\| \jump{p}^\aa_{\D_p} \|_{L^2((0,T)\times\Gamma)} \le C,
&\quad& \\
\frac{1}{\sqrt{M}}\|\Pi^m_{\D_p} p_m\|_{L^\infty(0,T;L^2(\Omega))} \le C, 
&\quad&
\max_{t\in [0,T]}\| \bu(t)\|_{\U_0} \le C,\\
\|d_{f,\D_\bu}\|_{L^\infty(0,T;L^4(\Gamma))}\le C, &\quad& \|\lambda\|_{L^2(0,T;{\Hminushalfn})}\le C.
\end{aligned}
\end{equation}
\label{th:energy_estimates}
\end{theorem}

\begin{remark}[Assumptions on fracture width and porosity]
Unlike those in \cite{GDM-poromeca-cont,GDM-poromeca-disc}, these estimates hold here without any assumption of lower bound on $\phi_\D$ or $d_{f,\D_\bu}$ (the latter being anyway bounded below by $d_0$ owing to its definition \eqref{GD_closures} and to Lemma \ref{lem:local.coulomb}). As a consequence, when invoking in the next section arguments similar to those in \cite{GDM-poromeca-cont,GDM-poromeca-disc}, we do not have to impose such lower bounds, and the arguments remain valid purely under the assumptions of Theorem \ref{th:convergence}.
\end{remark}

{
\begin{remark}[Time-dependent source term]
A time-dependent source term $\mathbf{f}$ could as well be considered without any major difficulty, provided the inertia term $\rho_0 \partial_{t}^2 \overline {\bm u}$ (with $\rho_0 > 0$ the material density) is added to the first equation of \eqref{eq_edp_meca} in order to establish the estimates in Theorem \ref{apriori.est}.
\end{remark}
}

\begin{theorem}[Existence of a discrete solution]\label{th:existence}
Under Assumptions~\ref{first.hyp}--\ref{last.hyp}, \eqref{eq:GD.coercif}, \eqref{infsup}, there exists at least one solution of the gradient scheme \eqref{eq:GS}. 
\end{theorem}

\subsection{Convergences of $d_{f,\D_\bu}$ and $\phi_{\D}$}

\begin{proposition}[Estimates on the time translates of $d_{f,\D_\bu}$ and $\phi_\D$]\label{prop_timetranslates}
Let $\tau,\tau'  \in (0,T)$ and, for $s\in (0,T]$, denote by $k_s$ the natural number such that $s\in (t_{k_s},t_{k_s+1}]$. For any $\psi_f \in X_{\D_p^f}$ and any $\psi_m\in X^0_{\D_p^m}$, it holds
\begin{multline}
    \label{est_timetranslates_df}
\Big|  \< d_{f,\D_{\bu}}(\tau) - d_{f,\D_{\bu}}(\tau'), \Pi^f_{\D_p} \psi_f \>_{L^2(\Gamma)} \Big| \\
\lsim  \sum_{k = k_{\tau}+1}^{k_{\tau'}}
 \dtk  \left( 
 \xi^{(1),k+1}_f \| \nabla^f_{\D_p} \psi_f \|_{L^8(\Gamma)} +
 \xi^{(2),k+1}_f \| \Pi^f_{\D_p} \psi_f \|_{L^2(\Gamma)} 
+ \,\dsp \sum_{\aa=\pm}  \xi^{(1),k+1}_\aa \| \jump{(0,\psi_f)}^\aa_{\D_p}\|_{L^2(\Gamma)} 
 \right),
\end{multline}
and
\begin{multline}
    \label{est_timetranslates_phi}
\Big|  \< \phi_{\D}(\tau) - \phi_\D(\tau'), \Pi^m_{\D_p} \psi_m \>_{L^2(\Omega)} \Big| \\
\lsim  \sum_{k = k_{\tau}+1}^{k_{\tau'}}
 \dtk  \left( 
 \xi^{(1),k+1}_m \| \nabla^f_{\D_p} \psi_m \|_{L^2(\Omega)} +
 \xi^{(2),k+1}_m \| \Pi^m_{\D_p} \psi_m \|_{L^2(\Omega)} 
+ \,\dsp \sum_{\aa=\pm}  \xi^{(1),k+1}_\aa \| \jump{(\psi_m,0)}^\aa_{\D_p}\|_{L^2(\Gamma)} 
 \right),
\end{multline}
with $(\xi^{(i),k+1}_{\omega})_{i=1,2;\omega=m,f,\pm;k\in\{0,\ldots,N\}-1}$ such that
\begin{equation}\label{eq:bound.RHS.timetranslate}
\sum^{N-1}_{k = 0} \dtk \[  \sum_{\omega=m,f}\left( \xi^{(2),k+1}_\omega \right)^2 + \sum_{\rt=m,f,\pm} \left( \xi^{(1),k+1}_{\rm rt} \right)^2 \]  \lsim 1.
\end{equation}
\end{proposition}

\begin{proof}
The proof is similar to that of \cite[Proposition 4.5]{GDM-poromeca-disc}, but we provide a few details for the sake of legibility. Take $\psi_f\in X^0_{\D_p^f}$, write $\langle d_{f,\D_{\bu}}(\tau),\Pi_{\D_p}^f\psi_f\rangle_{L^2(\Gamma)}-\langle d_{f,\D_{\bu}}(\tau'),\Pi_{\D_p}^f\psi_f\rangle_{L^2(\Gamma)}$ as the sum of the jumps at each time step $(t_{k})_{k\in\{k_{\tau}+1,\ldots,k_{\tau'}\}}$ between $\tau$ and $\tau'$ (we assume $\tau<\tau'$), and use the scheme \eqref{GD_hydro} with $\varphi=(0,\varphi_f)\in (X^0_{\D_p})^{N+1}$ where $\varphi^k_f=0$ if $k\not\in \{k_{\tau}+1,\ldots,k_{\tau'}\}$, $\varphi_f^k=\psi_f$ otherwise. This leads to
\begin{align*}
\langle d_{f,\D_{\bu}}(\tau){}&,\Pi_{\D_p}^f\psi_f\rangle_{L^2(\Gamma)}-\langle d_{f,\D_{\bu}}(\tau'),\Pi_{\D_p}^f\psi_f\rangle_{L^2(\Gamma)}=\sum_{k=k_\tau+1}^{k_{\tau'}}\dtk \int_\Gamma (\delta_t d_{f,\D_{\bu}})(t_{k+1})\Pi_{\D_p}^f\psi_f\d\sigma\\
={}&\sum_{k=k_\tau+1}^{k_{\tau'}}\Big(\int_{t_k}^{t_{k+1}} \int_\Gamma h_f \Pi_{\D_p}^f \psi_f ~\d t\d\sigma
-\int_{t_k}^{t_{k+1}}  \int_\Gamma  \dsp {d_{f,\D_\bu}^{\,3} \over 12 \eta} \,\, \nabla_{\D_p}^f p_f \cdot   \nabla_{\D_p}^f \psi_f  ~\d t\d\sigma \\
&   \qquad- \sum_{\aa \in \{+,-\}}
  \int_{t_k}^{t_{k+1}}  \int_\G   \Lambda_f \jump{p}^\aa_{\D_p}   \jump{(0,\psi_f)}^\aa_{\D_p}  ~\d t\d\sigma \Big).
\end{align*}
The estimate \eqref{est_timetranslates_df} follows writing $d_{f,\D_\bu}^{\,3}\nabla_{\D_p}^f p_f=d_{f,\D_\bu}^{\,\nf32}\nabla_{\D_p}^f p_f\times d_{f,\D_\bu}^{\,\nf32}$, applying Cauchy--Schwarz and generalised H\"older inequalities (the latter with exponents $(2,8/3,8)$) and setting  
\begin{align*}
  & \xi^{(1),k+1}_f = \| (d^{k+1}_{f,\D_{\bu}})^{\nf 3 2} \nabla^f_{\D_p} p^{k+1}_f \|_{L^2(\Gamma)}  \| d^{k+1}_{f,\D_{\bu}} \|^{\nf 3 2}_{L^4(\Gamma)},\quad   \xi^{(2),k+1}_f  = \Big\| {1\over \dtk}\int_{t_k}^{t_{k+1}} h_f(t,\cdot)~\d t \Big\|_{L^2(\Gamma)},\\
  & \xi^{(1),k+1}_\aa = \| \jump{p^{k+1}}^\aa_{\D_p}\|_{L^2(\Gamma)}.  
\end{align*}
The proof of \eqref{est_timetranslates_phi} is obtained in a similar way: writing $\langle \phi_\D(\tau),\Pi_{\D_p}^m\psi_m\rangle_{L^2(\Omega)}-\langle \phi_\D(\tau'),\Pi_{\D_p}^m\psi_m\rangle_{L^2(\Omega)}$ as the sum of its jumps, choosing $\varphi=(\varphi_m,0)$ with $\varphi_m=(0,\ldots,0,\psi_m,\ldots,\psi_m,0,\ldots,0)$ in the scheme \eqref{GD_hydro} and using Cauchy--Schwarz inequality, we obtain  \eqref{est_timetranslates_phi} with
$$
\xi_m^{(1),k+1}=\|\nabla_{\D_p}^mp_m^{k+1}\|_{L^2(\Omega)}\quad\mbox{ and }\quad \xi_m^{(2),k+1}=\Big\| {1\over \dtk}\int_{t_k}^{t_{k+1}} h_m(t,\cdot)~\d t \Big\|_{L^2(\Omega)}.
$$ 
The Korn and Sobolev trace inequalities  {(recalling that $\Gamma$ is a $(d-1)$-dimensional object)} show that, for all $t\in [0,T]$, $\|\bu(t)\|_{L^4(\Gamma)^d}\lesssim \|\bu(t)\|_{\U_0}$. Since $d_{f,\D_\bu}=d_0-\jump{\bu}_{n,\mathcal F}$, we infer from the bound on $\bu$ in \eqref{apriori.est} that 
\begin{equation}\label{eq:bound.df}
\max_{t\in[0,T]}\|d_{f,\D_\bu}(t)\|_{L^4(\Gamma)}\lesssim 1.
\end{equation}
Using then the other estimates in \eqref{apriori.est}, we deduce \eqref{eq:bound.RHS.timetranslate}.
\end{proof}

\begin{proposition}[Compactness of {$d_{f,\D_\bu}$ and $\phi_\D$}]
  \label{prop_compactnessdfsf}
  Up to a subsequence, as $l\to+\infty$:
  \begin{itemize}
  \item $(d_{f,\D^l_\bu})_{l\in\N}$ converges strongly in $L^\infty(0,T;L^p(\Gamma))$ for all $2 \leq p < 4$,
  \item $(\phi_{\D^l})_{l\in\N}$ converges uniformly-in-time weakly in $L^2(\Omega)$ (as per \cite[Definition C.14]{gdm}).
  \end{itemize}
\end{proposition}

\begin{proof}
From Proposition \ref{prop_timetranslates} and following the same arguments as in the proof of \cite[Proposition 4.8]{GDM-poromeca-disc} (with $\Pi_{\D_p}^fs_f^\alpha=1$ in this reference), we get the uniform-in-time weak-$L^2$ compactness of $d_{f,\D_\bu}$ and $\phi_{\D^l}$. Using the compactness of the trace $\U_0\to L^2(\Gamma)$ and the bound on $\bu$ in \eqref{apriori.est}, we obtain as in the proof of \cite[Proposition 4.10]{GDM-poromeca-cont} a uniform-in-time estimate on the $L^2(\Gamma)$-space-translates of $d_{f,\D_\bu}=d_0-\jump{\bu}_{n,\mathcal F}$ (recall that $d_0$ is continuous). In conjunction with the uniform-in-time weak-$L^2(\Gamma)$ compactness of $d_{f,\D_\bu}$ and \cite[Lemma A.2]{GDM-poromeca-cont}, this gives the compactness of $d_{f,\D_\bu}$ in $L^\infty(0,T;L^2(\Gamma))$. The conclusion then follows from the bound \eqref{eq:bound.df}.
\end{proof}

\subsection{Preliminary lemmas}

We establish here two density results and a compactness property that will be used in the proof of convergence.

\begin{lemma}[Density of smooth functions in $\U_0$]
Let $\bCinf{\Omega}{\ol{\G}}$ be the space of smooth functions $\O\backslash\overline\G\to\R^d$ that vanish on $\partial\Omega$ and whose derivatives of any order admit finite limits on each side of $\Gamma$. Then, $\bCinf{\Omega}{\ol{\G}}$ is dense in $\U_0$.
\label{lem:density.U0}
\end{lemma}

\begin{proof}
The tips of the fracture network $\G$ (including those on $\partial\Omega$) are finite sets of points (in 2D) or segments of lines (in 3D). In either case, their 2-capacity is zero. There exists therefore a sequence of functions $(z_\ell)_{\ell\in\N}$ in $H^1_0(\Omega)$ such that $0\le z_\ell\le 1$, $z_\ell=1$ on a neighborhood of the tips, and $z_\ell\to 0$ in $H^1_0(\Omega)$ as $\ell\to+\infty$ (see, for example, the proof of \cite[Lemma 2.3]{CDPRX17} for an explicit construction).

Consider now $\bv\in\U_0$. We have to approximate $\bv$ by functions in $\bCinf{\Omega}{\ol{\G}}$. As $R\to+\infty$, the function defined by truncating each component of $\bv$ at levels $\pm R$ converges to $\bv$ in $\U_0$; without loss of generality and relying on a diagonal argument, we can thus assume that $\bv$ is bounded.
The function $(1-z_\ell)\bv$ belongs to $H^1(\Omega\backslash\G)$, has a vanishing trace on $\partial\Omega$, and has a support that does not intersect the fracture tips. This last property enables us to take the convolution of $(1-z_\ell)\bv$ with kernels that are upwinded on each side of the fracture and on $\partial\Omega$ (without encountering any issue with the upwinding when the two sides of the fracture meet at a tip, since $(1-z_\ell)\bv$ vanishes on a neighborhood of the tip), to create a function $\bw_\ell\in \bCinf{\Omega}{\ol{\G}}$ such that $\|\bw_\ell-(1-z_\ell)\bv\|_{\U_0}\le 1/\ell$. Using $0\le z_\ell\le 1$, $z_\ell\to 0$ in $H^1_0(\Omega)$, and the fact that $\bv$ is bounded, we easily see that $(1-z_\ell)\bv\to \bv$ in $\U_0$. Hence, $\bw_\ell\to\bv$ in $\U_0$ and the proof is complete.
\end{proof}

{
\begin{lemma}[Density of smooth functions in $C_f$]
The set $C^0(\G;[0,+\infty))$ is dense in $C_f$ for the $\Hminushalfn$-topology.
\label{lem:density.Cf}
\end{lemma}

\begin{proof}

\underline{Step 1}: \emph{reduction to a single infinite fracture}.\\
Let $\mu\in C_f$ and take $(z_\ell)_{\ell\in\N}$ as in the proof of Lemma \ref{lem:density.U0}. For all $v\in\Hhalfn$ and all $\bv\in\U_0$ such that $v=\jump{\bv}_n$, since $z_\ell$ is continuous across $\Gamma$ we have $(1-z_\ell)v=\jump{(1-z_\ell)\bv}_n$, which shows that $(1-z_\ell)v\in\Hhalfn$ and, since $(1-z_\ell)\bv\to\bv$ in $\U_0$ as $\ell\to+\infty$ (see the proof of Lemma \ref{lem:density.U0}), that $(1-z_\ell)v\to v$ in $\Hhalfn$ as $\ell\to+\infty$.

Hence, we can define $\mu_\ell=(1-z_\ell)\mu\in\Hminushalfn$ classically by setting $\HGbracket{\mu_\ell}{v}=\HGbracket{\mu}{(1-z_\ell)v}$ for all $v\in\Hhalfn$. By the arguments above, $\mu_\ell\in C_f$ and $\mu_\ell\to\mu$ weakly-$\star$ in $\Hminushalfn$ as $\ell\to+\infty$. Since $\Hminushalfn$ is a Hilbert space, Mazur's lemma then ensures some {finite} convex combinations of $(\mu_\ell)_{\ell\in\N}$ (which remain in $C_f$ since this is a convex cone) converge to $\mu$ strongly in $\Hminushalfn$. We therefore only have to approximate each of these convex combinations to conclude.

Let $\omega$ be one of these convex combinations. We have $\omega=\alpha\mu$ for some $\alpha\in C^\infty(\R^d)$ convex combination of $(1-z_\ell)_{\ell\in\N}$. In particular, $\alpha$ vanishes in a neighborhood of the tips of $\Gamma$, and we can therefore find $(\varpi_i)_{i\in I}\in C^\infty_c(\Omega)^I$ such that the support of each $\varpi_i$ only intersects $\Gamma$ along $\Gamma_i$ (without touching its tips), and such that $\omega=\sum_i \varpi_i {\alpha}\mu$ (we just need to take $\varpi_i$ equal to $1$ on a neighborhood of $\mathrm{supp}(\alpha)\cap\Gamma_i$, and whose support does not intersect $\cup_{j\not=i}\ol{\Gamma}_j$).

Each $\varpi_i{\alpha}\mu$ is supported in a compact set of $\Gamma_i$, an open set of an hyperplane of $\R^d$, and we only need to approximate each $\varpi_i\alpha\mu$ to conclude. Trivially extending this form by $0$ outside its support, we can assume that $\Gamma$ is a full hyperplane of $\R^d$.
\medskip

\underline{Step 2}: \emph{proof for a single planar fracture}.\\
Having assumed that $\Gamma$ is a hyperplane, we easily check that $\Hhalfn=H^{\nf12}(\Gamma)$ and that its dual space is $H^{-\nf12}(\Gamma)$.
The proof in this case uses classical arguments, which we provide for the sake of completeness.

Let $(\rho_k)_{k\in\N}$ be symmetric smoothing kernels on the hyperplane $\Gamma$. The convolution $\mu*\rho_k:\Gamma\to\R$ is defined by $\mu*\rho_k(\x)=\HGbracket{\mu}{\rho_k(\cdot-\x)}$ and we have $\mu*\rho_k\in C^\infty_c(\Gamma)$; moreover, since $\mu\in C_f$ and $\rho_k\ge 0$ we also have $\mu*\rho_k\ge 0$, and thus $\mu*\rho_k\in C^0(\Gamma;[0,+\infty))$.

For all $v\in H^{\nf12}(\Gamma)$, we have $\HGbracket{\mu*\rho_k}{v}=\HGbracket{\mu}{v*\rho_k}$. Since $v*\rho_k\to v$ in $H^{\nf12}(\Gamma)$ as $k\to+\infty$, we infer that $\mu*\rho_k\to \mu$ weakly-$\star$ in $H^{-\nf12}(\Gamma)$. Using again Mazur's lemma, we infer that convex combinations of
$(\mu*\rho_k)_{k\in\N}$ (which therefore belong to $C^0(\Gamma;[0,+\infty))$) converge to $\mu$ strongly in $H^{-\nf12}(\Gamma)$, which concludes the proof.
\end{proof}
}

\begin{lemma}[Partial weak/strong compactness]\label{lem:ws.compactness}
Let $T>0$, $O$ be an open bounded subset of $\R^s$ and $(f_l)_{l\in\N}$, $(g_l)_{l\in\N}$ be two sequences of functions
in $L^2((0,T)\times O)$ such that
\begin{itemize}
\item as $l\to+\infty$, $f_l\weakto f$ and $g_l\weakto g$ in $L^2((0,T)\times O)$;
\item $(f_l)_{l\in\N}$ is strongly compact in space in the following sense: setting, for $\delta>0$,
$$
T(f_l,\delta):=\sup_{z\in\R^s,\,|z|\le \delta}\|f_l(\cdot,\cdot+z)-f_l\|_{L^2((0,T)\times O)}
$$
(where $f_l$ has been extended by $0$ outside $O$), we have
\begin{equation}\label{eq:Tdelta}
\sup_{l\in\N} T(f_l,\delta)\to 0\mbox{ as $\delta\to 0$};
\end{equation}
\item $(g_l)_{l\in\N}$ is strongly compact in time, weakly in space, in the following sense: for all $\chi\in C^\infty_c(O)$, setting
$$
G_{l,\chi}(t)=\int_O g_l(t,\x)\chi(\x)\d\x,
$$
the sequence $(G_{l,\chi})_{l\in\N}$ converges strongly in $L^2(0,T)$.
\end{itemize}
Then, for all $\psi\in C([0,T]\times\overline{O})$, as $l\to+\infty$,
$$
\int_0^T\int_O f_l(t,\x)g_l(t,\x)\psi(t,\x)~\d t\d\x\to \int_0^T\int_O f(t,\x)g(t,\x)\psi(t,\x)~\d t\d\x.
$$
\end{lemma}

\begin{proof}
The proof is carried out using the same technique as in \cite[Theorem 5.4]{droniou.eymard2016}, noticing that, in this theorem, the assumption that $(\nabla_{D_m}\zeta_m)_m$ is bounded only serves to obtain estimates on the space translates -- covered by \eqref{eq:Tdelta} here -- and that the assumption that $(\delta_m\beta_m)_m$ is bounded is only used in Step 3 to get the strong convergence of the equivalent of $G_{l,\zeta}$, which we have assumed here.
\end{proof}

\subsection{Proof of the convergence theorem}\label{sec:proof.convergence}

We can now prove Theorem \ref{th:convergence}. Owing to the estimates \eqref{apriori.est}, to the coercivity property \eqref{eq:GD.coercif}, and to Proposition \ref{prop_compactnessdfsf}, we can extract subsequences which converge as per \eqref{eq:convergences} (the convergence of $\phi_{\D^l}$ following from its expression in \eqref{GD_closures} -- note that if $M<\infty$ then $\Pi_{\D^l_p}^mp^l_m$ is bounded in $L^\infty(0,T;L^2(\Omega))$ and converges thus weakly-$\star$ in this space, and if $M=\infty$ then $\phi_{\D^l}$ does not depend on $\Pi_{\D^l_p}^mp^l_m$), but without identifications of the links between the various limits. 

The limit-conformity of $(\D_p^l)_{l\in\mathbb{N}}$ enables, as in \cite{GDM-poromeca-disc}, to adapt the lemma of regularity of the limit, classical in the GDM framework (see \cite[Lemma 4.8]{gdm} for standard parabolic models, \cite[Proposition 3.1]{BHMS2016} for a hybrid model with discontinuous pressures similar to the one considered here), and obtain the links between the limits \eqref{eq:conv.pm}, \eqref{eq:conv.grad.pm}, \eqref{eq:conv.pf} and \eqref{eq:conv.jump.p}. The identification of the limit \eqref{eq:conv.gradpf} is done as in \cite{GDM-poromeca-disc}, testing with a smooth function that is compactly supported far from the fracture tips, using the $L^2$-bound of $\nabla_{\D_p}^fp_f^l$ away from these tips, and the convergence \eqref{eq:conv.df}. Finally, the identification of the limits of $\phi_{\D^l}$ and $d_{f,\D_\bu^l}$ is straightforward from their definition in terms of $p_m^l$ and $\bu^l$, and from the convergence of these unknowns.

It remains to prove that $(\ol{p},\ol{\bu},\ol{\lambda})$ is a weak solution to \eqref{eq:model}. The proof that $(\ol{p},\ol{\bu})$ satisfies the flow equation \eqref{weak:flow} is done using the consistency of $(\D_p^l)_{l\in\N}$ as in \cite{GDM-poromeca-disc} (in a simpler way here, and with straightforward identification of the matrix--fracture flux term), and is therefore omitted. We instead focus on the mechanical parts of the equation, which contain the novelty in the form of the Lagrange multiplier and variational inequality accounting for contact.

We prove \eqref{weak:meca}--\eqref{weak:multiplier} through the equivalent formulation \eqref{weak:meca.alternate}--\eqref{weak:multiplier.alternate}.
By Lemma~\ref{lem:density.U0} and \cite[Corollaire 1.3.1]{poly_intsob}, {linear combinations of functions of the form $(t,\x)\mapsto\theta(t)\ol{\bw}(\x)$ (with $\theta\in C^\infty_c(0,T)$ and $\ol{\bw}\in \bCinf{\Omega}{\ol{\G}}$) are dense in $L^2(0,T;\U_0)$; moreover, all the terms in \eqref{weak:meca.alternate} are continuous with respect to $\ol{\bv}$ for the topology of $L^2(0,T;\U_0)$. We therefore only have to prove this relation for functions of the form $\bv(t,\x)=\theta(t)\bw(\x)$ with $\theta$ and $\bw$ as above.} The consistency of $(X^l_{\D_\bu})_l$ ensures the existence of $\bw^l\in (X^0_{\D_\bu^m})^l$ such that $\bw^l\to \ol{\bw}$ in $\U_0$. Testing \eqref{GD_meca} with $\bv=\int_{t_{k-1}}^{t_k}\theta(t)~\d t\bw^l$, summing over $k\in\{1,\ldots,N\}$ and dropping the index $l$ we have
\begin{multline*}
     \int_0^T\int_\Omega \( \bbsig(\bu) : \bbeps(\theta\bw)  
    - b ~\Pi_{\D_p}^m p_m~  \div\,(\theta\bw)\)  ~\d t\d\x 
   +  \int_0^T\int_\Gamma   \lambda ~ \jump{\theta\bw}_n~\d t\d\sigma\\
     + \int_0^T \int_\Gamma  \Pi_{\D_p}^f p_f~  \jump{\theta\bw}_{n,\faces} ~\d t\d\sigma 
  = \int_0^T\int_\Omega  \mathbf{f} \cdot  (\theta\bw) ~\d t\d\x .
\end{multline*}
The strong convergence $\bw\to\ol{\bw}$ in $\U_0$ (which implies the strong convergence $\jump{\bw}_n\to\jump{\ol{\bw}}_n$ in ${\Hhalfn}$) and the weak convergences \eqref{eq:conv.pm}, \eqref{eq:conv.pf}, \eqref{eq:conv.u} and \eqref{eq:conv.lambda} enable us to pass to the limit and see that \eqref{weak:meca.alternate} holds for $\ol{\bv}=\theta\ol{\bw}$, and thus for all $\ol{\bv}\in L^\infty(0,T;\U_0)$. 

We now consider the variational inequality \eqref{weak:multiplier.alternate}. 
Using Lemma \ref{lem:density.Cf} and an easy adaptation of \cite[Corollaire 1.3.1]{poly_intsob} (to manage the fact that $L^2(0,T;C_f)$ is valued in a subset of the vector space ${\Hminushalfn}$, instead of the entire vector space), we see that any $\ol{\mu}\in L^2(0,T;C_f)$ can be approximated in this space by functions in $C^0([0,T]\times \Gamma;[0,+\infty))$. We thus only have to consider the case where $\ol{\mu}$ belongs to the latter set. Let $\mu^l\in (C_{\D_f})^{N}$ be defined by $(\mu^l)^k_\sigma=\frac{1}{|\sigma|}\int_\sigma \ol{\mu}(t_k,\x)\d\sigma(\x)$. The continuity of $\ol{\mu}$ ensures that $\mu^l\to \ol{\mu}$ in $L^2((0,T)\times \Gamma)$ as $l\to+\infty$. Moreover, using $(\mu^l)^k$ in \eqref{GD_meca_var}, recalling that $\lambda\jump{\bu}_n=0$ (by Lemma \ref{lem:local.coulomb}), multiplying by $\dtkmun$ and summing over $k\in\{1,\ldots,N\}$ we have
$$
\int_0^T\int_\G \mu^l\jump{\bu^l}_n\le 0.
$$
We can then pass to the limit $l\to+\infty$, using the weak convergence of $\jump{\bu^l}_n$ in $L^2((0,T)\times \G)$ which comes from \eqref{eq:conv.u}, to get
\begin{equation}\label{eq:var.limit.1}
\int_0^T\HGbracket{\ol{\mu}}{\jump{\ol{\bu}}_n}\le 0.
\end{equation}
This proves {\eqref{eq:var.limit.1}} for nonnegative continuous functions $\ol{\mu}$ (for which the bracket is actually an integral), and thus for all $\ol{\mu}\in L^2(0,T;C_f)$. To conclude the proof of convergence for the mechanical equations, it remains to show that
\begin{equation}\label{eq:var.limit.toshow}
\int_0^T\HGbracket{\ol{\lambda}}{\jump{\ol{\bu}}_n}\ge 0.
\end{equation}
Indeed, using $\ol{\mu}=\ol{\lambda}$ in \eqref{eq:var.limit.1} then shows that \eqref{eq:var.limit.toshow} holds with an equality which, combined with \eqref{eq:var.limit.1} for a generic $\ol{\mu}$, yields \eqref{weak:multiplier.alternate}.

\medskip

The rest of the proof is devoted to establishing \eqref{eq:var.limit.toshow}. Making $\bv=\bu^k$ in \eqref{GD_meca}, using Lemma \ref{lem:local.coulomb} to cancel the term involving $\lambda^k\jump{\bu^k}_n$, multiplying by $\dtkmun$ and summing over $k\in\{1,\ldots,N\}$ we find
\begin{equation}\label{eq:var.limit.2}
\begin{aligned}
 \int_0^T \int_\Omega \bbsig(\bu) : \bbeps(\bu)  ~\d t\d\x 
    -\int_0^T \int_\Omega b ~\Pi_{\D_p}^m p_m~  \div\,\bu  ~\d t\d\x 
     &+  \int_0^T\int_\Gamma  \Pi_{\D_p}^f p_f~  \jump{\bu}_{n,\faces} ~\d t\d\sigma\\ 
  ={}& \int_0^T\int_\Omega  \mathbf{f} \cdot  \bu ~\d t\d\x.
\end{aligned}
\end{equation}
We aim at passing to the inferior limit on each of these terms. The weak convergence \eqref{eq:conv.u} and the fact that $\bbsig(\bu) : \bbeps(\bu)  =\frac{E}{1+\nu}\(\bbeps(\bu):\bbeps(\bu) + \frac{\nu}{1-2\nu} (\div\,\bu)^2\)$ give
\begin{equation}\label{eq:var.limit.3}
\begin{aligned}
\int_0^T \int_\Omega \bbsig(\ol{\bu}) : \bbeps(\ol{\bu}) ~\d t\d\x\le{}&\liminf_{l\to+\infty}\int_0^T \int_\Omega \bbsig(\bu) : \bbeps(\bu)  ~\d t\d\x ,\\
\int_0^T\int_\Omega  \mathbf{f} \cdot  \ol{\bu} ~\d t\d\x={}&\lim_{l\to+\infty}\int_0^T\int_\Omega  \mathbf{f} \cdot  \bu ~\d t\d\x.
\end{aligned}
\end{equation}
By \eqref{apriori.est}, $\bu$ takes its values in the ball $B_{\U_0}(C)$ in $\U_0$ centred at 0 and of radius $C$. Let $\mathcal K$ be the image under $\jump{{\cdot}}_n$ of $B_{\U_0}(C)$. By the Sobolev-trace inequality $\mathcal K$ is bounded in $L^4(\G)$ and relatively compact in $L^s(\G)$ for all $s<4$. Hence, the Kolmogorov theorem shows that the $L^2$-translations along $\G$ of $\jump{\bu(t)}_n$ are uniformly equi-continuous (with respect to $l$ and $t$), and thus that the face averages of these functions are uniformly close to the functions:
$$
{\sup_{t\in(0,T)}}\left\|\jump{\bu^l(t)}_{n,\faces}-\jump{\bu^l(t)}_n\right\|_{L^2(\G)}\to 0\mbox{ as $l\to+\infty$}.
$$
Since the sequence $(\jump{\bu^l(t)}_{n,\faces})_l=(d_0-d_{f,\D^l_\bu})_l$ is relatively compact in $L^2((0,T)\times \G)$, see Proposition \ref{prop_compactnessdfsf}, this proves that $(\jump{\bu^l}_n)_l$ is also relatively compact in the same space, and thus converges in this space to $\jump{\ol{\bu}}_n$. The weak convergence \eqref{eq:conv.pf} then shows that, as $l\to+\infty$,
\begin{equation}\label{eq:var.limit.4}
\int_0^T\int_\Gamma  \Pi_{\D_p}^f p_f~  \jump{\bu}_{n,\faces} ~\d t\d\sigma\to
\int_0^T\int_\Gamma  \ol{p}_f~  \jump{\ol{\bu}}_n ~\d t\d\sigma.
\end{equation}
To analyse the convergence of the remaining term in \eqref{eq:var.limit.2} (the second one), we use \eqref{GD_closures} to re-write it as
\begin{equation}\label{eq:var.limit.5}
\begin{aligned}
-\int_0^T \int_\Omega b ~\Pi_{\D_p}^m p_m~  {}&\div\,\bu  ~\d t\d\x=
\int_0^T\int_\Omega \frac1M (\Pi_{\D_p}^mp_m)^2  ~\d t\d\x
-\int_0^T\int_\Omega \phi_\D \Pi_{\D_p}^m p_m  ~\d t\d\x\\
&-\int_0^T \int_\Omega b ~\Pi_{\D_p}^m p_m~  \div\,\bu^0  ~\d t\d\x
+\int_0^T\int_\Omega \Pi_{\D_p}^m\phi_m^0 \Pi_{\D_p}^m p_m ~\d t\d\x\\
&-\int_0^T\int_\Omega \frac1M \Pi_{\D_m}^mp_m^0\Pi_{\D_p}^mp_m ~\d t\d\x.
\end{aligned}
\end{equation}
The consistency of the GDs gives the strong convergence of all the terms corresponding to initial conditions (see \cite[Appendix 2]{GDM-poromeca-disc} on how to handle the initial condition on the displacement), and we can therefore, using the weak convergence \eqref{eq:conv.pm}, pass to the limit in the last three terms. The same weak convergence also gives
$$
\int_0^T\int_\Omega \frac1M (\ol{p}_m)^2 ~\d t\d\x \le\liminf_{l\to+\infty}\int_0^T\int_\Omega \frac1M (\Pi_{\D_p}^mp_m)^2 ~\d t\d\x.
$$
To pass to the limit in the second term on the right-hand side of \eqref{eq:var.limit.5}, we invoke Lemma \ref{lem:ws.compactness} with $f_l=\Pi_{\D^l_p}p_m^l$ and $g_l=\phi_{\D^l}$. The strong compactness in space of $(\Pi_{\D^l_p}p_m^l)_l$, requested to apply this lemma, follows from the bound on $(\nabla_{\D^l_p}p_m^l)_l$ in \eqref{apriori.est} and from the assumption \eqref{eq:compact.matrix} (together with Kolmogorov's compactness theorem), while the strong compactness in time, weakly in space, of $(\phi_{\D^l})_l$ follows from Proposition \ref{prop_compactnessdfsf} and from the definition of uniform-in-time weak-$L^2$ convergence. Taking the inferior limit of \eqref{eq:var.limit.5} and using the continuous closure law \eqref{weak:closure.phi} we infer
$$
\liminf_{l\to+\infty}\left(-\int_0^T \int_\Omega b ~\Pi_{\D_p}^m p_m~  \div\,\bu  ~\d t\d\x\right)
\ge -\int_0^T \int_\Omega b ~\ol{p}_m~  \div\,\ol{\bu}  ~\d t\d\x.
$$
Combined with \eqref{eq:var.limit.3} and \eqref{eq:var.limit.4}, this enables us to take the inferior limit of \eqref{eq:var.limit.2} (using the fact that the $\liminf$ of a sum is larger than the sum of the $\liminf$) to find
\begin{equation}\label{eq:var.limit.6}
 \int_0^T \int_\Omega \bbsig(\ol{\bu}) : \bbeps(\ol{\bu})  ~\d t\d\x 
    -\int_0^T \int_\Omega b ~\ol{p}_m~  \div\,\ol{\bu}  ~\d t\d\x 
     +  \int_0^T\int_\Gamma  \ol{p}_f~  \jump{\ol{\bu}}_n ~\d t\d\sigma 
  \le \int_0^T\int_\Omega  \mathbf{f} \cdot  \ol{\bu} ~\d t\d\x.
\end{equation}
Setting $\ol{\bv}=\ol{\bu}$ in \eqref{weak:meca.alternate} and comparing with \eqref{eq:var.limit.6}, we deduce that \eqref{eq:var.limit.toshow} holds, which concludes the proof of Theorem \ref{th:convergence}.

\section{Other models}
\label{sec:other.models}

We briefly describe here how the previous convergence analysis can be adapted to other poromechanical models, when accounting for contact, or highlight the specific challenges some of them raise.

\paragraph{Fracture-width-dependent normal transmissibility.} As already noticed, the normal fracture transmissibility $\Lambda_f$ has been assumed to be independent of $\ol{d}_f$ for simplicity, but the above analysis readily extends to a fracture-width-dependent transmissibility, provided that the function $\ol{d}_f\mapsto\Lambda_f(\ol{d}_f)$ is continuous and bounded above and below by strictly positive constants.

\paragraph{More general schemes for the mechanics.} In \cite{GDM-poromeca-disc}, a more general discretization of the mechanical part of the model is considered using gradient discretizations (as done for the flow here, but adapted to elasticity equations). These GDs cover some nonconforming methods, and stabilised schemes \cite{DL14}. 

Under the same assumptions on these mechanical GDs as in \cite{GDM-poromeca-disc}, replacing the continuous jump and norms in $\U_0$ by the reconstructed jump and norm of the GD
and assuming that the inf--sup condition \eqref{infsup} holds, the analysis above applies.

\paragraph{Continuous pressure.} The model \eqref{eq:model} with discontinuous interface pressures is inspired by the two-phase flow model of \cite{GDM-poromeca-disc}. A continuous-pressure model, {corresponding to the case $\Lambda_f \rightarrow +\infty$,} is presented in \cite{GDM-poromeca-cont}; the treatment of this model requires to be able to ``localise'' the test functions in order to separate the matrix from the fracture terms in the time translates estimates (see the proof of \cite[Proposition 4.7]{GDM-poromeca-cont}); such a localisation is not required with a discontinuous pressure model as the fracture and matrix translate estimates \eqref{est_timetranslates_df} and \eqref{est_timetranslates_phi} are already separated.

The localisation of test functions is achieved through the introduction of a \emph{cut-off property} on the GDs $(\D_p^l)_{l\in\N}$ (see \cite[Section 3.1]{GDM-poromeca-cont}). Under this additional assumption, the convergence analysis carried out here for \eqref{eq:model} can be adapted to the model with continuous pressure.

\paragraph{Two-phase flows.} The two references \cite{GDM-poromeca-cont,GDM-poromeca-disc} mentioned above concern two-phase flow models, but without contact. In this situation, the existence of a solution to the discrete or continuous model is not proved, and the convergence analysis can only be done under the assumption that discrete solutions exist and that the fracture aperture (and matrix porosity) remain bounded below by strictly positive quantities; this is a requirement as the continuous model itself does
not account for any mechanism that ensures this property.

A two-phase {flow} discontinuous-pressure model with contact and friction is proposed in \cite{BDMP:21}, and served as ground for the single-phase model \eqref{eq:model} with contact (and no friction). Even putting aside the friction in the model, the analysis of the two-phase flow model involves more technical arguments, in particular to manage the nonlinear dependency between the saturations and the capillary pressure. However, the most challenging aspect is the dependency of the \emph{a priori} estimates on a lower bound of the porosity (see \cite[Lemma 4.3]{GDM-poromeca-disc} and \cite[Theorem 4.2]{BDMP:21}, and compare with the absence of such a lower bound in Theorem \ref{th:energy_estimates} above). If such a lower bound is assumed, then the convergence analysis of Section \ref{sec:convergence} could probably be adapted to the two-phase flow models without friction (and with the open question of identifying the limit of interface fluxes in case of the discontinuous-pressure model, see the discussion in \cite[Section 4.4]{GDM-poromeca-disc}).

An alternative to assuming a lower bound on the discrete porosities is to use a variant of the model with a semi-frozen porosity as detailed, e.g., in \cite[Remarks 3.1 and 4.3]{BDMP:21}. This model ensures uniform estimates just assuming that the initial discrete porosity (a datum of the model) is bounded below by a strictly positive number. However, this model involves the product of the discrete saturation with $\delta_t\phi_\D$, whose convergence cannot be handled as in \cite[Proof of Theorem 4.1]{GDM-poromeca-disc} by moving the discrete time derivative onto the test function, as this would also create a discrete time derivative of the discrete saturation whose strong convergence cannot be ensured.

{Another alternative developed in \cite{both2023global} for a poromechanical model without fractures is to introduce a bound on the porosity in the model itself in order to ensure its positivity. The extension to mixed-dimensional models remains to be investigated.}

\paragraph{Friction.} Including friction in the contact model consists in adding to \eqref{weak:multiplier} a term of the form $\langle \ol{\boldsymbol{\mu}}_\tau-\ol{\boldsymbol{\lambda}}_\tau,\jump{\partial_t\ol{\bu}}_\tau\rangle_\Gamma$, with $ \jump{\ol{\bv}}_\tau = \jump{\ol{\bv}} -\jump{\ol{\bv}}_n \n^+$, see \cite{BDMP:21}. This term does not challenge the \emph{a priori} estimates on the discrete solutions, but creates severe difficulties for passing to the limit. Indeed, following the arguments in Section \ref{sec:proof.convergence} would require the strong convergence of the discrete versions $\jump{\partial_t\bu}_\tau$. Even if an inertial term $\rho\partial_t^2\ol{\bu}$ is added to the mechanical equations, to ensure that some estimates on the (discrete or continuous) time derivatives of the jumps of the displacement can be established, this would not suffice to ensure their strong compactness. {In \cite{dehoop2023coupling}, the authors account for a regularized frictional contact model with normal stiffness which admits a primal variational formulation. The analysis is based on an additional viscous term in the mechanics constitutive law and a frozen fracture conductivity to deal with the absence of a lower bound on the fracture aperture.} 
 
\section{Numerical example}
\label{sec:num.example}
We present in this section numerical results obtained in a two-dimensional framework, to investigate the convergence of the discrete solutions to the coupled model~\eqref{eq:model}.
\subsection{Space and time discretizations}
The flow component~\eqref{eq_edp_hydro} is discretized using a mixed-dimensional HFV scheme~\cite{ecmor-joubine,BHMS2016} (which fits the GDM framework), whereas for the  contact mechanics component~\eqref{eq_edp_meca} we use, along the lines of~\cite{BDMP:21}, a mixed $\bb P_2$--$\bb P_0$ formulation for displacement and Lagrange multipliers (normal stresses) on fracture faces, respectively (cf.~Figure~\ref{fig:mesh_hfv_p2}). Concerning the flow part, notice also that HFV schemes support general polytopal meshes and anisotropic matrix permeabilities, unlike TPFA schemes such as in~\cite{BDMP:21}. 

With a view towards studying convergence in time, we choose a uniform time step $\Delta t$. 
At each time step, a Newton--Raphson algorithm is used to compute the flow unknowns.
On the other hand, the contact mechanics is solved using a non-smooth Newton method, formulated in terms of an active-set algorithm. For a more detailed description, we refer the reader to~\cite{BDMP:21}. Finally, the coupled nonlinear system is solved at each time step using a fixed-point method on cell pressures $p_m$ and fracture-face pressures $p_f$, accelerated by a Newton--Krylov algorithm~\cite{ecmor}.

We study the convergence rates in space and time, for the flow and mechanics unknowns, by computing the $L^2$-norm of the error in the matrix and along the fracture network, using in the latter case $\bb P_2 $ reconstructions for the jump of the displacement field $\jump{\bu}$ and Lagrange multiplier $\lambda$, as in~\cite{BDMP:21}.

\subsection{Single-phase flow in a rectangular domain including six fractures}
\label{sec:bergen_fluid}
As a reference example, we consider the test case presented in~\cite[Section 4.3]{contact-norvegiens}. In~\cite{BDMP:21}, a similar test case was considered (from~\cite[Section 4.1]{contact-norvegiens}), including friction effects, but without taking into account the effect of a Darcy flow.
It consists of a $2\,m\times 1\,m$ domain -- the long side being aligned with the $x$-axis, the short one with the $y$-axis -- including a network $\Gamma = \bigcup_{i=1}^6 \Gamma_i$ of six hydraulically active fractures, cf.~Figure~\ref{fig:bergen_fluid}. Fracture 1 is made up of two sub-fractures forming a corner, whereas one of the tips of fracture 5 lies on the boundary of the domain.

Here, we consider some major modifications with respect to~\cite[Section 4.3]{contact-norvegiens}, concerning the model employed, the data, as well as initial and boundary conditions.
First, while the fractures are considered impervious in~\cite{contact-norvegiens}, here we allow for flow across and along them. Second, to fully exploit the capabilities of the HFV flow discretization, we consider the following anisotropic permeability tensor in the matrix:
$$\bb  K_m = K_m \,\mathbf e_x \otimes \mathbf e_x + \frac{K_m}{2}\, \mathbf e_y \otimes \mathbf e_y,$$
$\mathbf e_x$ and $\mathbf e_y$ being the unit vectors associated with the $x$- and $y$-axes, respectively.

The permeability coefficient is set to $K_m = 10^{-15}$~m$^2$, the Biot coefficient to $b=0.8$, the Biot modulus to $M=10$~GPa, the dynamic viscosity to $\eta = 10^{-3}$~Pa${\cdot}$s. The initial matrix porosity is set to $\phi_m^0 = 0.4$, and the fracture aperture corresponding to both contact state and zero displacement field is given by 
$$d_0(\x) = \delta_0 \frac{\sqrt{\arctan(a D_i(\x))}}{\sqrt{\arctan(a \ell_i)}},\quad \x\in\Gamma_i,\quad i\in\{1,\ldots,6\},$$
where $D_i(\x)$ is the distance from $\x$ to the tips of fracture $i$, $\delta_0 = 10^{-4}$~m, $a=25$~m$^{-1}$, and $\ell_i$ is a fracture-dependent characteristic length: it is equal to $L_i/2$ ($L_i$ being the length of fracture $i$) if fracture $i$ is immersed, to $L_i$ if one of its ends lies on the boundary, and to the distance of a corner from tips, if it includes a corner. Note that the above expression behaves asymptotically as $\sqrt{D_i(\x)}$ when $\x$ is close to fracture tips, which is in agreement with~\cite[Remark 3.1]{GWGM2015}.

The normal transmissibility of fractures is  set to the fixed value $\Lambda_f = {\delta_0 \over 6\eta}$. Note that a fracture-aperture-dependent normal transmissibility $\Lambda_f = {d_f \over 6 \eta}$ could also be used without any significant effect on the solution, due to its very large value compared with the matrix conductivity. 
The initial pressure in the matrix and fracture network is $p_m^0 = p_f^0 = 10^5$~Pa. Notice that the initial fracture aperture differs from $d_0$, since it is computed by solving the mechanics given the initial pressures $p_m^0$ and $p_f^0$, and therefore the contact state along fractures, represented in Figure~\ref{fig:bergen_fluid}, is such that most fractures are open at $t=0$. Finally, we use the same values of Young's modulus and Poisson's ratio as in~\cite{contact-norvegiens}, i.e.~$E=4$~GPa and $\nu = 0.2$. The final time is set to $T=2000$~s.

Concerning boundary conditions, for the flow, all sides are assumed impervious, except the left one, on which a pressure equal to the initial value $10^5$~Pa is prescribed. For the mechanics, the left and right boundaries are free, the bottom boundary is clamped; on the other hand, the top boundary is given a time-dependent boundary condition defined by
$$
\bu(t,\x) =
\begin{cases}
[0.005\,\mathrm{m}, -0.0005\,\mathrm{m}]^\top\,4t/T &\quad \text{if } t \le T/4,\\
[0.005\,\mathrm{m}, -0.0005\,\mathrm{m}]^\top &\quad \text{otherwise}.
\end{cases}
$$

The physical interpretation of this boundary condition has been already explained in~\cite{contact-norvegiens}: since the absolute value of the $y$-component of the top-boundary displacement increases linearly with time during the first quarter of the simulation, the domain undergoes compression, which gives rise to an increasing fluid pressure. For $t > T/4$, the top-boundary displacement is constant, which implies a decrease in the fluid pressure owing to the pressure Dirichlet condition on the left side, where it is kept equal to the initial value $10^5$~Pa. Note that the simulation time $T$ is set to the order of the pressure relaxation time. 

No analytical solution is available for this test case. We therefore investigate the convergence of our discretization by computing a reference solution on a fine mesh made of 182\,720 triangular elements for spatial approximation errors, and with a time step $\Delta t = 100/16$~s for time approximation errors. Figure~\ref{fig:pressures} shows the pressure variations in the matrix and in the fracture network with respect to the initial values, $\Delta p_\omega = p_\omega - p_\omega^0$, $\omega\in\{m,f\}$, obtained with the above mesh and a time step $\Delta t = 100$~s. The drain behavior of the fractures is clearly visible, reflecting the quasi constant pressure along each fracture as a result of the large fracture tangential conductivity, and the pressure continuity at matrix--fracture interfaces due to the large fracture normal transmissibility.
Figure~\ref{fig:bergen_df_phim_pm_pf} shows the time histories of the average matrix porosity, fracture aperture, and pressures, obtained with the finest mesh. It reflects the expected increase of fluid pressures until the peak of the top-boundary displacement at $t = T/4 = 500$~s, followed by the pressure relaxation for $t > T/4$ almost back to the initial values. As a result of the pressure decrease, it has been checked that the fracture aperture also decreases for $t > T/4$ almost back to the solution obtained for the uncoupled mechanical model at fixed initial pressures and maximum top-boundary displacement.  Figure~\ref{fig:bergen_u_lambda} shows the normal jump $\jump{\bu}_n$ and the tangential jump $\jump{\bu}_\tau = \jump{\bu} -\jump{\bu}_n \n^+$ of the displacement, and Lagrange multiplier, along fractures, at $t=T/4$.
We note in particular the singularity in the neighborhood of the corner in fracture 1  clearly reflected in the behavior of the Lagrange multiplier.
These behaviors are consistent with the contact state at fractures, shown in Figure~\ref{fig:bergen_fluid}, which is mainly governed by the Dirichlet boundary condition on the displacement field at the top boundary. 

To evaluate the convergence of our method, we compute the $L^2$ errors between the reference solution and the solutions obtained using three sequentially and uniformly refined meshes, consisting of 2\,855, 11\,420, and 45\,680 elements. Figure~\ref{fig:bergen_errors1} displays the convergence for the jumps of the normal and tangential  displacements and for the Lagrange multiplier along fractures at time $t=T/4$.
A convergence rate of $1.5$ is observed for the jumps on all fractures, while, as expected, the convergence rate observed for the Lagrange multiplier strongly depends on the singularity of the solution at the fracture tips and corner shown in Figure~\ref{fig:bergen_u_lambda}. Finally, in Figure~\ref{fig:bergen_errors2} we study the convergence rate of the matrix porosity, fracture aperture, and pressures with respect to space and time refinements. We compute the $L^2$ errors in time and space between the reference average solution and solutions obtained by refining from 20 to 320 time steps, fixing a space mesh of 11\,420 elements, and the $L^2$ error in space at final time, using the same set of meshes mentioned above, fixing the time step to $\Delta t = 100$~s. Asymptotic convergence rates equal to 1.5 and 1 are observed, respectively, in space and time.
\section{Conclusions}
\label{sec:conclusions}
We carried out a thorough numerical analysis, in the general framework of the Gradient Discretization Method, for a model coupling single-phase Darcy flows and contact mechanics in linear elastic fractured porous media, assuming frictionless contact and allowing for fluid pressure discontinuities at \mbox{matrix--fracture} interfaces. Two important features of the model are (i)~the nonlinear poromechanical coupling, yielding the fracture conductivity in terms of the fracture aperture \emph{via} the Poiseuille law; and (ii)~the degeneracy of fracture conductivity, given by the vanishing fracture aperture at immersed fracture tips. The major result of this contribution is Theorem~\ref{th:convergence}, giving the convergence of a discrete weak solution towards a continuous one by compactness techniques, and thereby yielding at the same time the existence of such a solution. In Section~\ref{sec:other.models}, we discussed how such an approach could be adapted to more complex models involving flow--mechanics coupling in fractured porous media. Finally, we performed in Section~\ref{sec:num.example} a numerical experiment in a two-dimensional setting, to investigate the convergence of our discretization both in space and in time. For this, we used a Hybrid Finite Volume (HFV) method for the flow and a mixed $\bb P_2$--$\bb P_0$ formulation for the contact mechanics; both schemes fit the Gradient Discretization framework.
%
%
\begin{figure}
\centering
\subfloat[Two-phase flow unknowns, $\alpha\in\{\l,\g\}$]{
\includegraphics[scale=1.05]{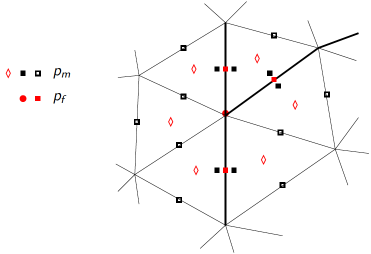}}\qquad
\subfloat[Contact mechanics unknowns]{
\includegraphics[scale=1.1]{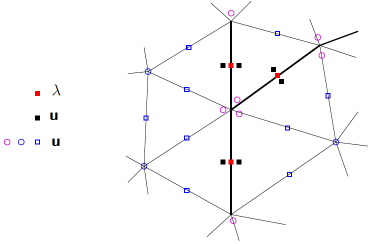}}
\caption{{Example of triangular mesh with three fracture edges in bold. The discrete unknowns are presented for the mixed-dimensional HFV discretization of the flow model in (a), and for the $\bb P_2$--$\bb P_0$ mixed discretization of the contact mechanics in (b). The discontinuities of the pressures and of the displacement are captured at matrix--fracture interfaces. Note that a nodal fracture pressure unknown is directly eliminated (hence not represented) when shared by two fracture edges, while it is kept as an additional discrete unknown when shared by three or more fracture edges}.}
\label{fig:mesh_hfv_p2}
\end{figure}
\begin{figure}
\centering
\subfloat[Contact state at $t=0$]{
\includegraphics[keepaspectratio=true,scale=.135]{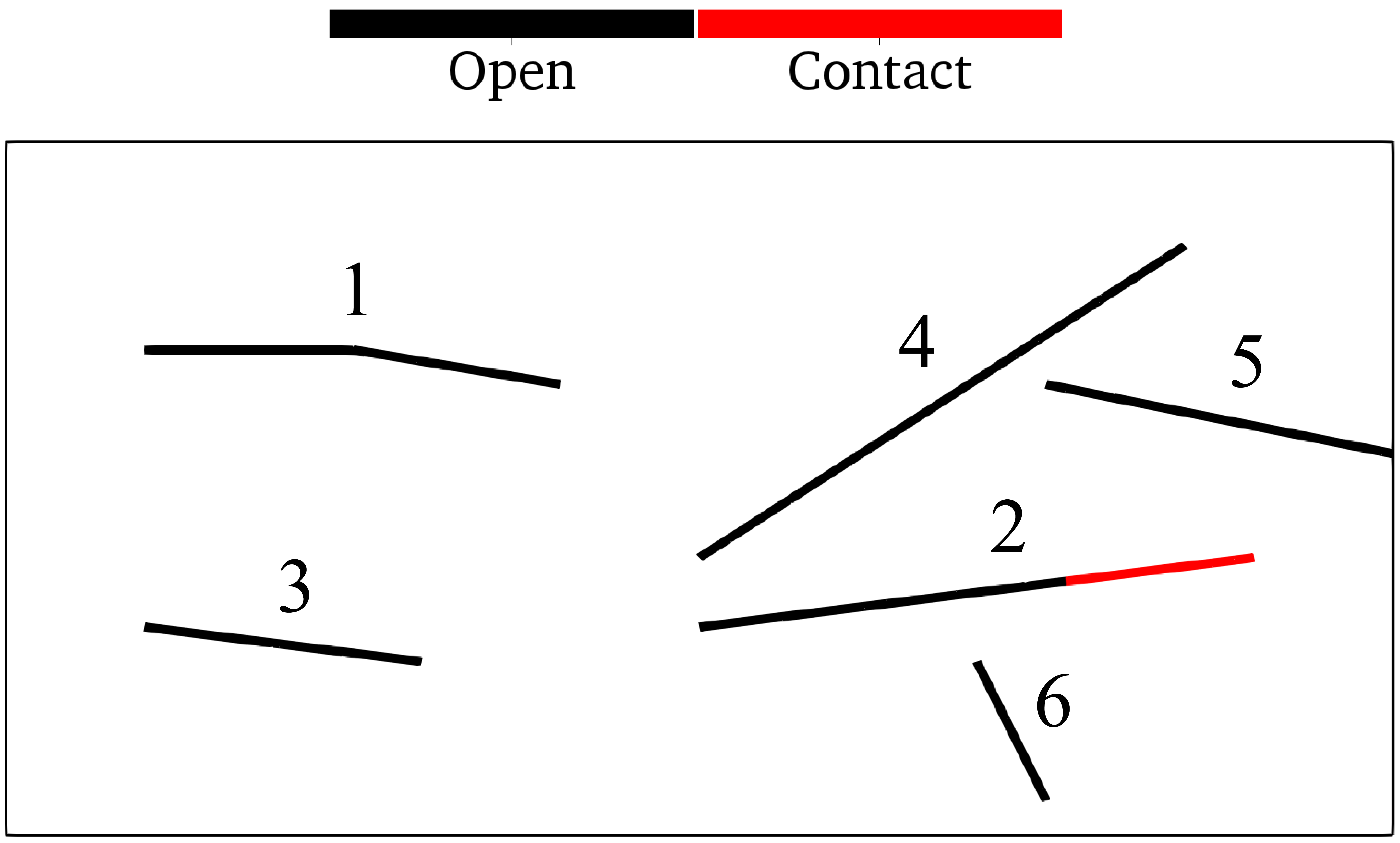}} \ \
\subfloat[Contact state at $t\ge T/4$]{
\includegraphics[keepaspectratio=true,scale=.135]{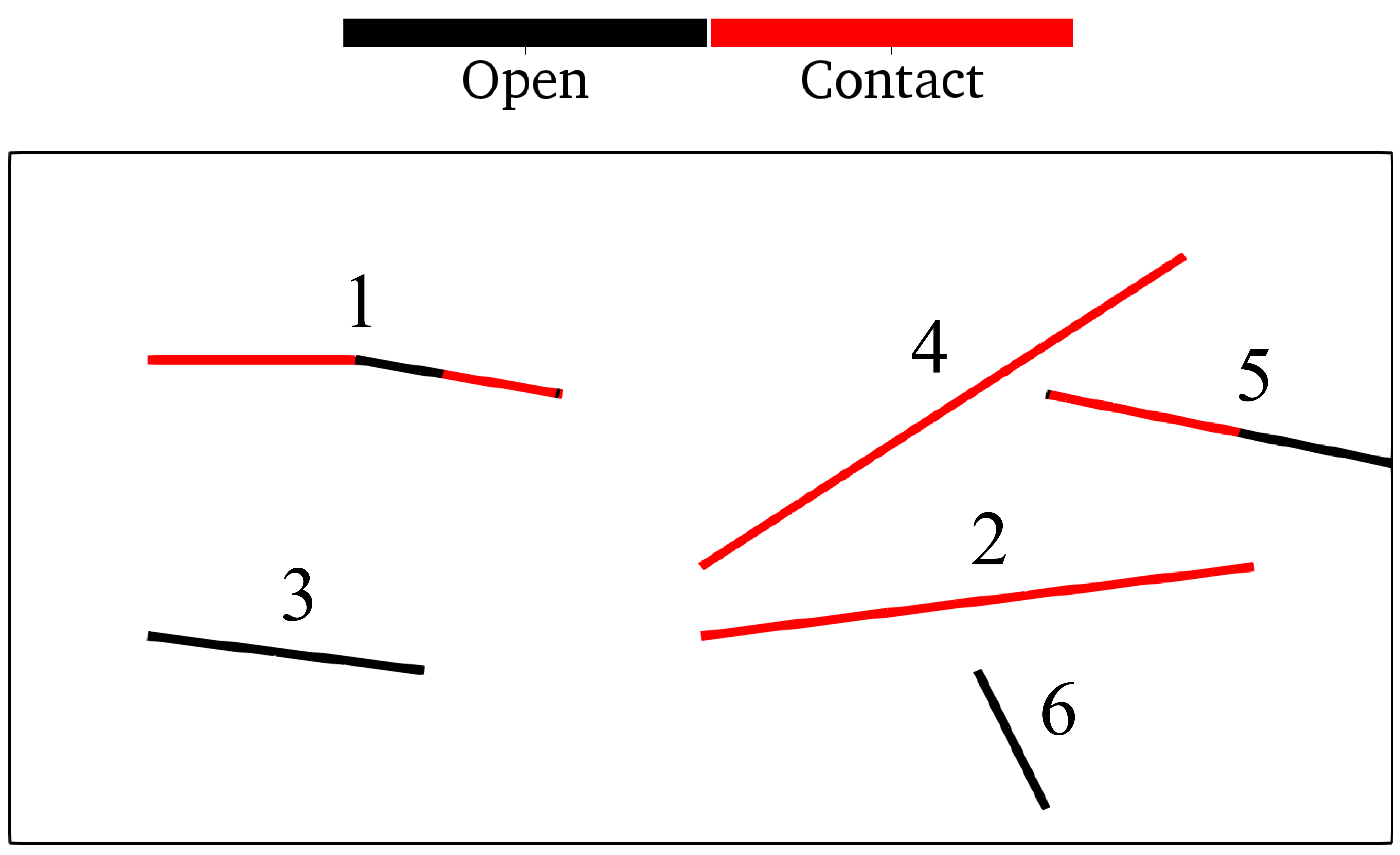}}
\caption{Two-dimensional, $2\times1$\,m domain containing six fractures. Fracture 1 comprises two sub-fractures making a corner, and fracture 5 has a tip on the boundary. The contact state along each fracture over time is also shown.}
\label{fig:bergen_fluid}
\end{figure}
\begin{figure}
\centering
\subfloat[$t=T/4$]{
\includegraphics[keepaspectratio=true,scale=.14]{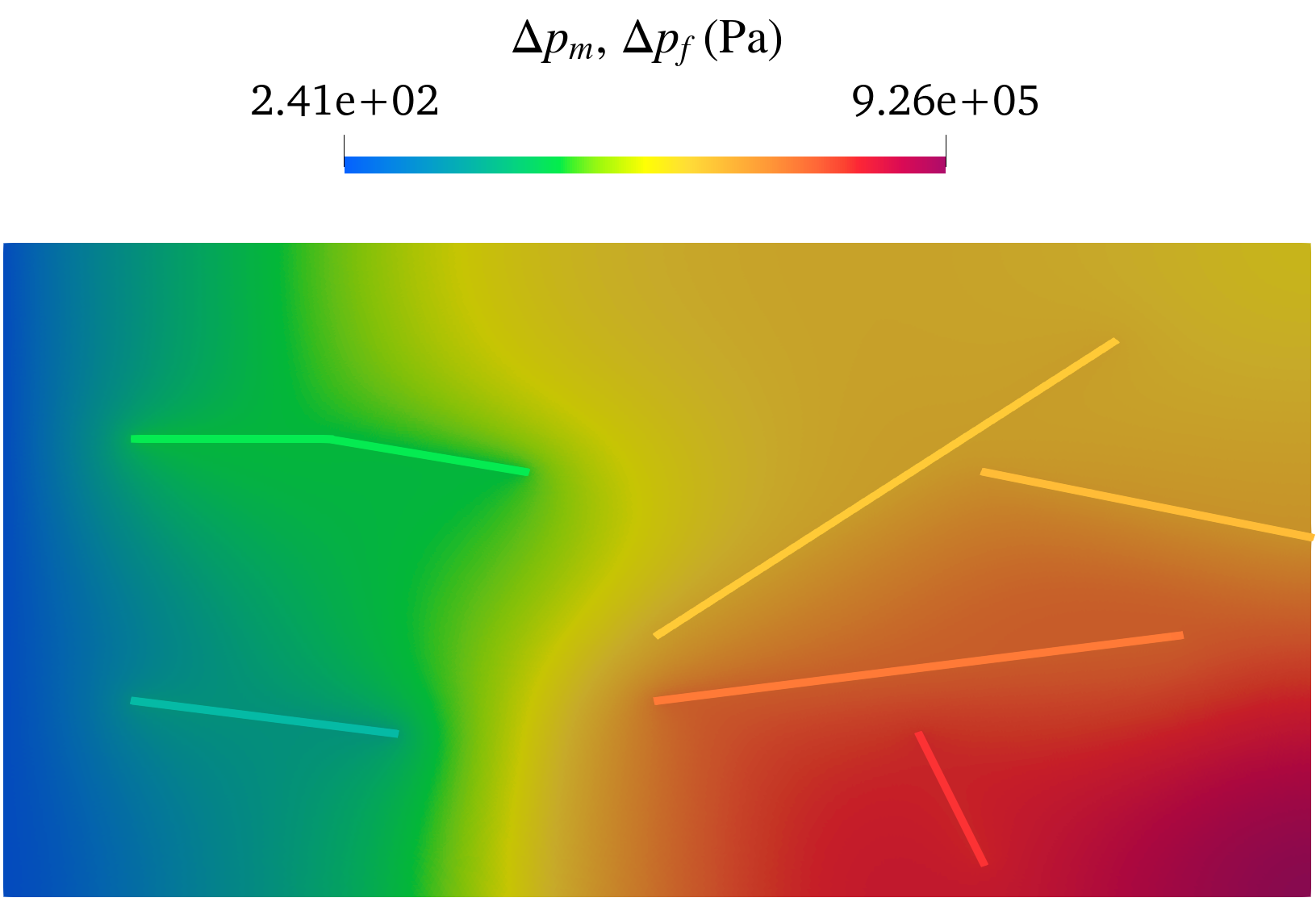}}
\subfloat[$t=T$]{
\includegraphics[keepaspectratio=true,scale=.135]{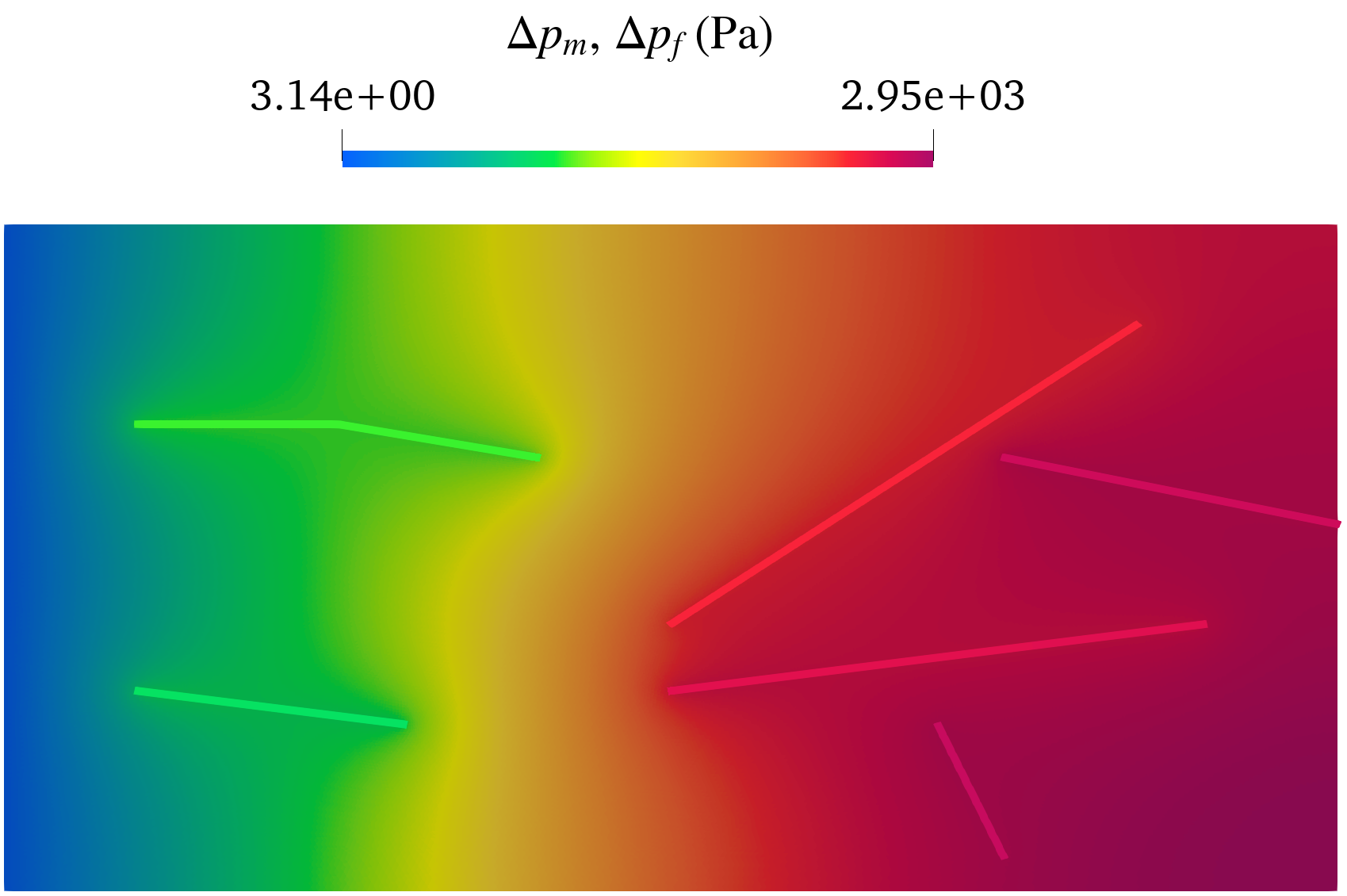}}
\caption{Variation of the matrix and fracture pressures (with respect to the initial value {$p_m^0=p_f^0=~10^5$~Pa}) at $t=T/4$ and $t=T$.}
\label{fig:pressures}
\end{figure}

\begin{figure}
\centering
\subfloat[Mean porosity $\phi_m$]{
\includegraphics[keepaspectratio=true,scale=.565]{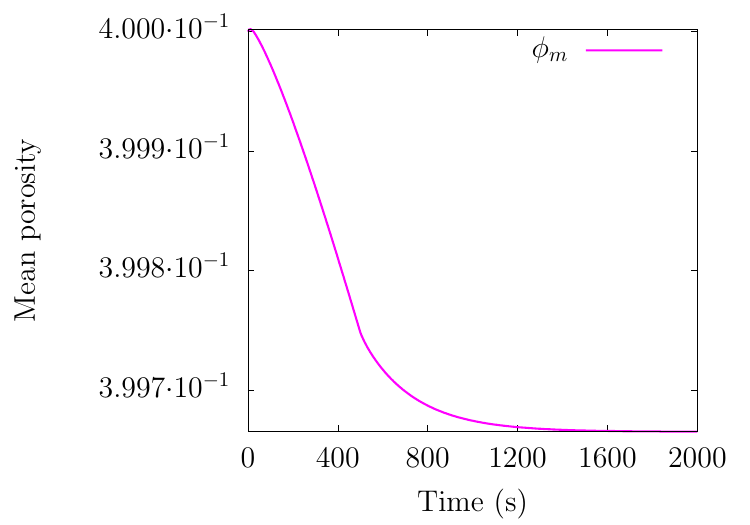}}
\subfloat[Mean fracture aperture $d_f$]{
\includegraphics[keepaspectratio=true,scale=.565]{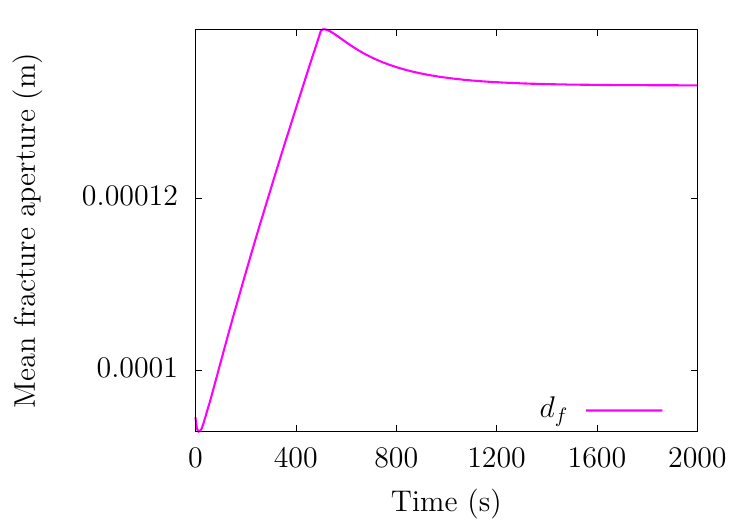}}\\
\subfloat[Mean matrix and fracture pressures $p_m$ and $p_f$]{
\includegraphics[keepaspectratio=true,scale=.565]{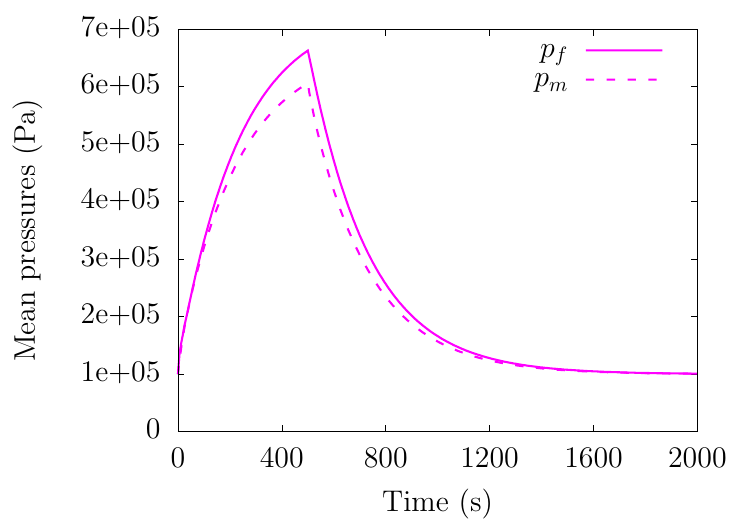}}
\caption{Mean fracture aperture, porosity, and matrix/fracture pressures vs.~time.}
\label{fig:bergen_df_phim_pm_pf}
\end{figure}

\begin{figure}
\centering

\subfloat[$-\jump{\bu}_n$ vs.~$\tau/L_i$]{
\includegraphics[keepaspectratio=true,scale=.625]{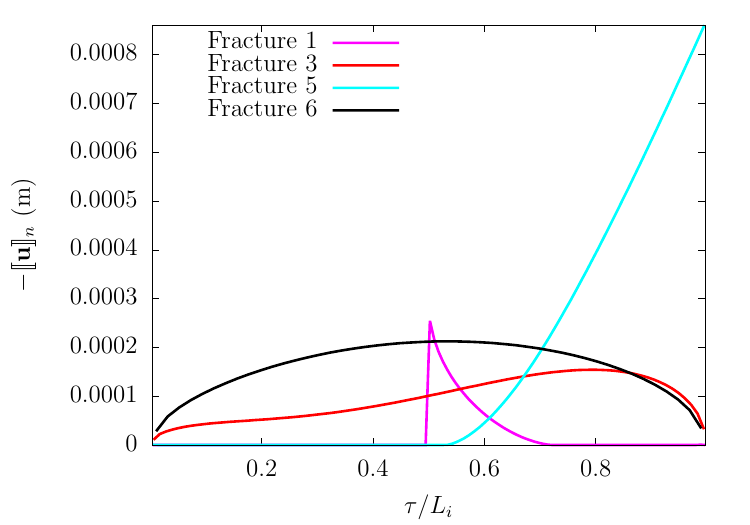}}
\subfloat[$\jump{\bu}_\tau$ vs.~$\tau/L_i$]{
\includegraphics[keepaspectratio=true,scale=.625]{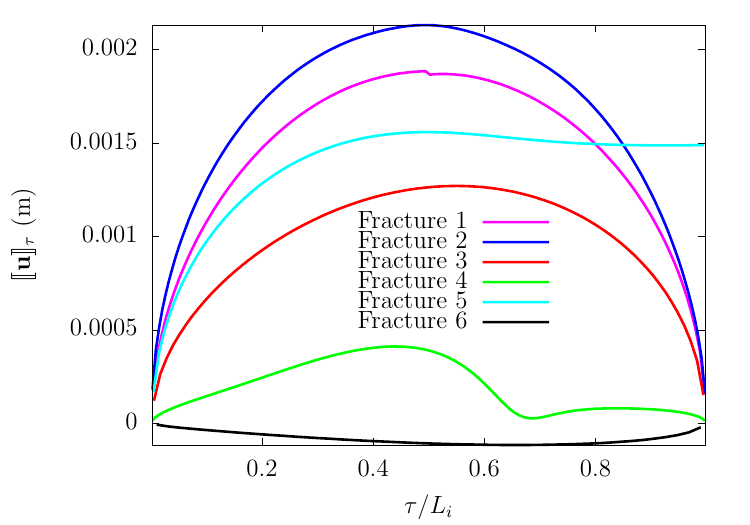}} \\
\subfloat[$\lambda$ vs.~$\tau/L_i$]{
\includegraphics[keepaspectratio=true,scale=.625]{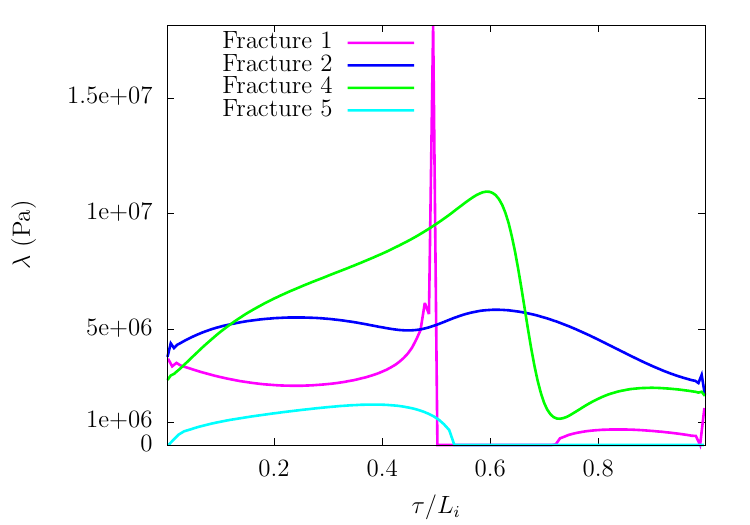}}

\caption{Normal $\jump{\bu}_n$ and tangential $\jump{\bu}_\tau$ jumps of the displacement field, and Lagrange multiplier at $t=T/4$ vs.~scaled distance from the tips ($L_i$ = length of fracture $i$) for each of the six fractures.}
\label{fig:bergen_u_lambda}
\end{figure}

\begin{figure}
\centering
\subfloat[$-\jump{\bu}_n$ vs.~$h$]{
\includegraphics[keepaspectratio=true,scale=.625]{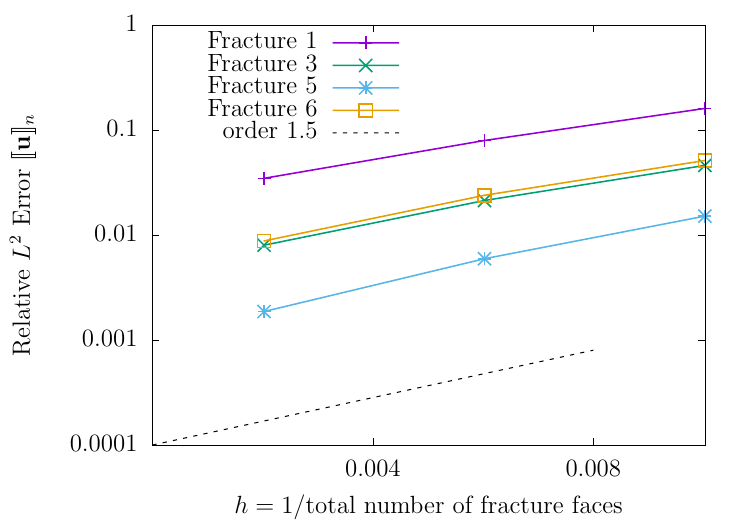}}
\subfloat[$\jump{\bu}_\tau$ vs.~$h$]{
\includegraphics[keepaspectratio=true,scale=.625]{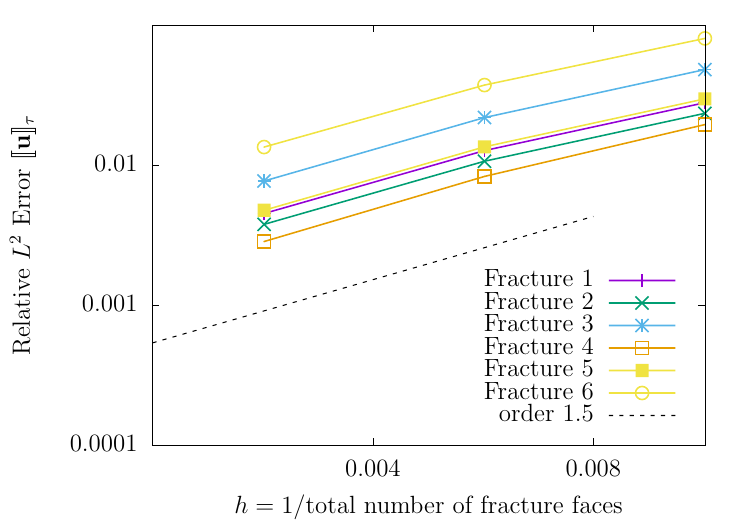}}\\[-.5ex]
\subfloat[$\lambda$ vs.~$h$]{
\includegraphics[keepaspectratio=true,scale=.625]{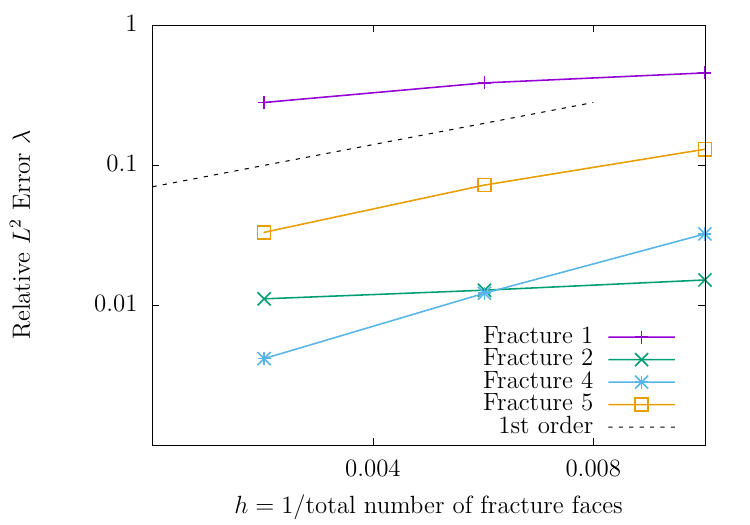}}\\
\caption{$L^2$ errors at $t=T/4$ between the current and reference solutions in terms of $\jump{\bu}_n$ (a), $\jump{\bu}_\tau$ (b), and $\lambda$ (c).}
\label{fig:bergen_errors1}
\end{figure}

\begin{figure}
\centering
\subfloat[Errors vs.~time step with the fixed mesh of 11\,420 cells \\ and the reference solution obtained with $\Delta t = 100/16$~s]{
\includegraphics[keepaspectratio=true,scale=.625]{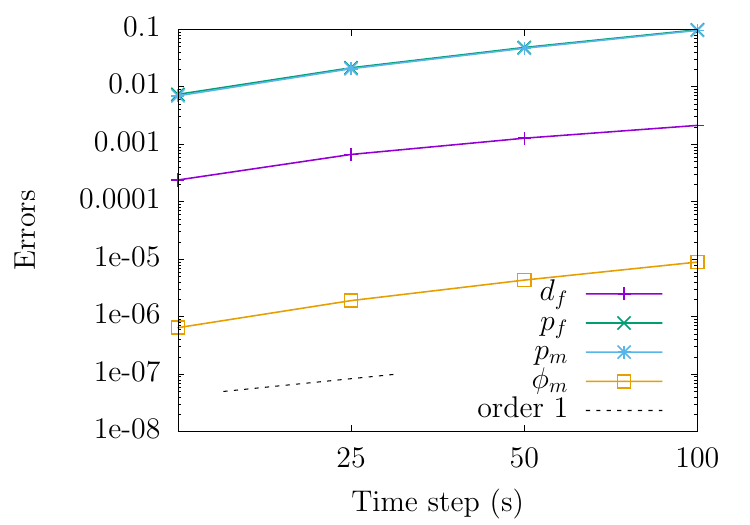}}
\subfloat[Errors vs.~meshsize at final time with fixed time step \\ $\Delta t = 100$~s and reference solution obtained on the finest mesh]{
\includegraphics[keepaspectratio=true,scale=.625]{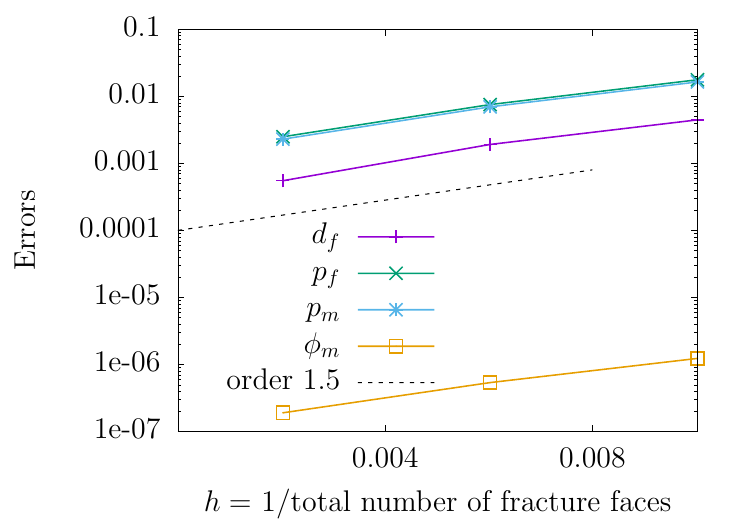}}
\caption{$L^2$ errors in time and space between the current and reference average solutions vs.~time step size, computed with a space mesh of 11\,420 elements and refining from 20 to 320 time steps (reference solution corresponding to the finest time step  $\Delta t = 100/16$~s) (a), and vs.~meshsize, computed with a fixed time step $\Delta t = 100$~s (b).}
\label{fig:bergen_errors2}
\end{figure}
\begin{acknowledgements}
Partially funded by the European Union (ERC Synergy, NEMESIS, project number 101115663).
Views and opinions expressed are however those of the authors only and do not necessarily reflect those of the European Union or the European Research Council Executive Agency. Neither the European Union nor the granting authority can be held responsible for them.
\end{acknowledgements}
%
%
%
\bibliographystyle{plain}
\bibliography{Poromeca_contact}
\end{document}